\documentclass{amsart}

\usepackage{quiver}

\usepackage{amsrefs}
\usepackage{hyperref}
\usepackage{amsfonts}
\usepackage{amssymb}
\usepackage{amsmath}
\usepackage{amsthm}
\usepackage[inline]{enumitem}
\usepackage{tikz-cd}
\usepackage{color}

\usepackage{todonotes}

\newtheorem{theorem}{Theorem}[section]
\newtheorem{corollary}[theorem]{Corollary}

\newtheorem{definition}[theorem]{Definition}
\newtheorem{convention}[theorem]{Convention}
\newtheorem{lemma}[theorem]{Lemma}

\newtheorem{fact}[theorem]{Fact}
\newtheorem{remark}[theorem]{Remark}

\newtheorem{question}[theorem]{Question}
\newtheorem{conjecture}[theorem]{Conjecture}

\begin{document}

\title{The Flat Cover Conjecture for monoid acts}

\author{Sean Cox}
\email{scox9@vcu.edu}
\address{
Department of Mathematics and Applied Mathematics \\
Virginia Commonwealth University \\
1015 Floyd Avenue \\
Richmond, Virginia 23284, USA 
}

\date{\today}

\thanks{Partially supported by NSF grant DMS-2154141.  We thank James Renshaw for clarifying some results from \cite{MR3249868} and \cite{MR1899443}, and Jun Maillard and Leonid Positselski for their comments during the author's Charles University Algebra Seminar on this topic in May 2025.}


\begin{abstract}

We prove that the Flat Cover Conjecture holds for the category of (right) acts over any right-reversible monoid $S$, provided that the flat $S$-acts are closed under stable Rees extensions.  The argument shows that the class $\mathcal{F}$-Mono ($S$-act monomorphisms with flat Rees quotient) is cofibrantly generated in such categories, answering a question of Bailey and Renshaw.  But cofibrant generation of $\mathcal{SF}$-Mono ($S$-act monomorphisms with \emph{strongly} flat Rees quotient) appears much stronger, since we show it implies that there is a bound on the size of the indecomposable strongly flat acts.  Similarly, cofibrant generation of $\mathcal{U}_{\mathcal{F}}$ (unitary monomorphisms with flat complement) implies a bound on the size of indecomposable flat acts.  The key tool is a new characterization of cofibrant generation of a class of monomorphisms in terms of ``almost everywhere" effectiveness of the class.
\end{abstract}

\maketitle


\section{Introduction}\label{sec_intro}

A classic theorem of Bass says that every $R$-module has a projective cover if and only if the ring $R$ is perfect.  Enochs conjectured in the 1980's that every $R$-module has a \emph{flat} cover, for any ring $R$.  This became known as the \emph{Flat Cover Conjecture} (FCC), and was proved around the year 2000 (\cite{MR1832549}).  Subsequent work of many authors proved analogues of the FCC in other additive categories, often with the motivation of doing homological algebra relative to the flat objects (instead of the projectives), e.g., \cite{MR2578515}, \cite{MR4121092}, \cite{MR3649814}, \cite{MR2673740}.  

The question of whether the FCC holds in \emph{non-additive} categories has proved more difficult.\footnote{It is further complicated by the fact that ``geometric flat object" (tensoring with it preserves monomorphisms) and ``categorical flat object" (being a directed colimit of finitely presented projectives), while equivalent in $R$-Mod for every ring $R$, are generally different in non-additive categories.  In categories like Act-$S$, ``flat" denotes the geometric version, while ``strongly flat" denotes the categorical version.}  For example, regarding the category Act-$S$ of (nonempty, right-) $S$-acts over a monoid $S$, Bailey and Renshaw~\cite[Corollary 4.14]{MR3251751} prove that FCC holds in Act-$S$ \emph{if} there is a bound on the size of the indecomposable flat $S$-acts.  This holds, for example, if $S$ is right-cancellable, in which case there is even a bound on the size of the indecomposable torsion-free acts \cite[Theorem 5.4]{MR3079784}.  But there are not always bounds on sizes of indecomposable flat acts (\cite[Example 5.13]{MR3251751}), and whether FCC holds for general $S$ is open.  In \cite{MR3249868} they mention the closely related problem of whether the class
\[
\mathcal{F}_S \text{-Mono}:= \{ \text{S-act monomorphisms with flat Rees quotient} \}
\]
is \emph{cofibrantly generated} (i.e., whether it is the closure of a \textbf{set} of morphisms under pushouts, transfinite compositions, and retracts).   The corresponding class is always cofibrantly generated in module categories (Rosick\'y~\cite{MR1925005}), but Bailey-Renshaw state that the proof there ``seems to depend on the additive structure of the category of modules", and they mention two problems:

\begin{question}[\cite{MR3249868}]
Is $\mathcal{F}_S$-Mono cofibrantly generated in Act-$S$?  The question implicitly assumes $S$ is right-reversible, because otherwise $\mathcal{F}_S$-Mono is empty (\cite[Corollary 4.9]{MR0829389}).
\end{question}

\begin{question}
[\cite{MR3249868}] Suppose $S$ is a monoid with a left zero.  Is $\mathcal{F}_S$-Mono cofibrantly generated in the category Act$_0$-$S$ of centered $S$-acts (acts with a unique fixed point)?
\end{question}

Theorem \ref{thm_Fmono_cofibgen} below answers both questions affirmatively, as long as $\mathcal{F}_S$-Mono is closed under composition.  Because Act-$S$ lacks an initial object (which is useful for deriving covers), we will often work instead with $\overline{\text{Act}}$-$S$ that includes the empty act as an initial object. In $\overline{\text{Act}}$-$S$, $\mathcal{F}_S$-Mono also includes morphisms of the form $\emptyset \to F$ with $F$ nonempty and flat.  

\begin{theorem}\label{thm_Fmono_cofibgen}

The following are equivalent for any monoid $S$.
\begin{enumerate}
 \item\label{item_S_rr} $S$ is right-reversible, and $\mathcal{F}_S$-Mono is closed under compositions in $\overline{\text{Act}}$-$S$.

 \item\label{item_F_Mono_cofibGen} $\mathcal{F}_S$-Mono is cofibrantly generated in $\overline{\text{Act}}$-$S$.
\end{enumerate}

\noindent If $S$ has a left zero (which implies right-reversibility), the following are equivalent:

\begin{enumerate}
 \item $\mathcal{F}_S$-Mono is closed under compositions in Act$_0$-$S$.
 \item $\mathcal{F}_S$-Mono is cofibrantly generated in Act$_0$-$S$.
\end{enumerate}

\end{theorem}

Closure of $\mathcal{F}_S$-Mono under compositions in $\overline{\text{Act}}$-$S$ is equivalent to requiring that flat $S$-acts are closed under stable Rees extensions, which means that if $A$ is a stable\footnote{\emph{Stable subact} is a weakening of \emph{pure subact}.} subact of the $S$-act $B$, $A$ is flat, and the Rees quotient $B/A$ is flat, then $B$ is flat.  This additional assumption seems to also be needed for the proof of part 2 of \cite[Theorem 3.11]{MR3249868}.  Theorem \ref{thm_Fmono_cofibgen}, together with an argument of Rosick\'y (see Lemma \ref{lem_CellGenExpandedGivesCovers}), yields:

\begin{corollary}\label{cor_Positive}
Suppose $S$ is a right-reversible monoid, and
flat $S$-acts are closed under stable Rees extensions.  Then FCC holds in Act-$S$.  If, in addition, $S$ has a left zero, then the FCC holds in Act$_0$-$S$.
\end{corollary}

We do not know how restrictive is the requirement that $\mathcal{F}_S$-Mono is closed under composition in $\overline{\text{Act}}$-$S$ (equivalently, flat $S$-acts are closed under stable Rees extensions in 
Act-$S$).  It holds, for example, over monoids where every quasi-flat act is flat (by \cite[Lemma 4.7]{MR0829389}), or when every stable monomorphism is pure (by \cite[Theorem IV.1.5]{renshaw1986flatness}).  The word \emph{stable} in ``flat $S$-acts are closed under stable Rees extensions" is really superfluous, since any nontrivial instance of a flat Rees quotient implies stability of the inclusion; see Lemma \ref{lem_Renshaw}.

While closure of $\mathcal{F}_S$ under stable Rees extensions is trivially necessary to get cofibrant generation of $\mathcal{F}_S$-Mono, we conjecture that this assumption can be eliminated from the statement of Corollary \ref{cor_Positive}.  See Section \ref{sec_Questions} for a discussion.  We can also achieve the FCC for Act-$S$ if we replace the assumption of closure of flats under stable Rees extensions with an assumption about $S$ that is stronger than right-reversibility:

\begin{theorem}\label{thm_LO}
Suppose $S$ is a monoid that is right-LO, i.e., for every $s,t \in S$, there is a $u \in S$ such that either $su=t$ or $tu=s$.  Then:
\begin{enumerate}
 \item The class
 \[
 \mathcal{F}\text{-PureMono}:= \left\{ \text{pure monomorphisms with flat Rees quotient} \right\} 
 \]
 is cofibrantly generated in $\overline{\text{Act}}$-$S$.
  \item The FCC holds in Act-$S$.
\end{enumerate}
\end{theorem}

We turn now to the notion of \emph{strongly flat} ($\mathcal{SF}$) acts and the corresponding question of whether every $S$-act has an $\mathcal{SF}$-cover; if so, we will say that the \emph{Strongly Flat Cover Conjecture} (SFCC) holds for Act-$S$.  Kruml~\cite{MR2432817} gives an example of a particular (non-right-reversible) monoid where the SFCC fails.\footnote{Kruml uses \emph{flat} to denote what is more commonly called \emph{strongly flat}.  See \cite{MR3547004} for a clarification.}  Bailey and Renshaw~\cite{MR3547004} ask for a characterization of those monoids where SFCC holds.  We are unable to answer this question, but we can shed light on when $\mathcal{SF}$-Mono is cofibrantly generated (a sufficient condition for SFCC to hold).  Cofibrant generation of $\mathcal{SF}$-Mono appears much stronger than cofibrant generation of $\mathcal{F}$-Mono, because the former implies a bound on the cardinalities of the indecomposable strongly flat acts, and left-collapsibility (not just right-reversibility) of $S$:

\begin{theorem}\label{thm_SF}
For any monoid $S$, the following are equivalent:
\begin{enumerate}
 \item\label{item_SF_MonoCofibGen} $\mathcal{SF}$-Mono is cofibrantly generated in $\overline{\text{Act}}$-$S$.

 \item\label{item_SF_LeftCollapseIndec} $S$ is left-collapsible, $\mathcal{SF}$-Mono is closed under compositions in $\overline{\text{Act}}$-$S$, and there is a bound on the size of the indecomposable strongly flat $S$-acts.

\end{enumerate}
\end{theorem}

As in the flat case of Theorem \ref{thm_Fmono_cofibgen}, we do not know how restrictive is the assumption that $\mathcal{SF}$-Mono is closed under compositions.  By \cite[Theorem 6.2(7)]{MR1899443}, $\mathcal{SF}$-Mono is the same as the class of ``P-unitary" monomorphisms with strongly flat quotient.  The assumption that $\mathcal{SF}$-Mono is closed under compositions in $\overline{\text{Act}}$-$S$ is equivalent to requiring that $\mathcal{SF}$ is closed under P-unitary Rees extensions.

An act monomorphism $f: A \to B$ is \textbf{unitary} if $B\setminus \text{im}(f)$ is closed under the action of $S$ on $B$.  It is proved in \cite[Theorem 3.5]{MR3249868} that the FCC holds in Act-$S$ if and only if 
\[
\mathcal{U}_{\mathcal{F}}:= \left\{ f: A \to B \ : \ f \text{ is unitary and } B \setminus \text{im}(f) \in \mathcal{F}  \right\}
\]
forms the left part of a \emph{Weak Factorization System}.  We are unable to characterize when this happens, but we can characterize exactly when $\mathcal{U}_{\mathcal{F}}$ satisfies the stronger property of being cofibrantly generated:
\begin{theorem}\label{thm_UnitaryFlat}
Let $S$ be a monoid and $\mathcal{X}$ a class of $S$-acts such that for any collection $(X_i)_{i \in I}$ of $S$-acts,
\[
\bigsqcup_{i \in I} X_i \in \mathcal{X} \ \iff \ \forall i \in I \ X_i \in \mathcal{X}.
\]
Then $\mathcal{U}_{\mathcal{X}}$ is cofibrantly generated in $\overline{\text{Act}}$-$S$ if, and only if, there is a bound on the cardinality of the indecomposable members of $\mathcal{X}$.
\end{theorem}

Both $\mathcal{F}$ and $\mathcal{SF}$ always satisfy the disjoint union property displayed in the statement of Theorem \ref{thm_UnitaryFlat}.  Hence:
\begin{corollary}
For any monoid $S$, $\mathcal{U}_{\mathcal{F}}$ is cofibrantly generated in $\overline{\text{Act}}$-$S$ if and only if there is a bound on the size of the indecomposable members of $\mathcal{F}$ (and similarly for $\mathcal{SF}$).
\end{corollary}

The proofs of all these results go through the key Theorem \ref{thm_NewChar} (page \pageref{thm_NewChar}), a new ``top-down" characterization of cofibrant generation of a class $\mathcal{M}$ of monomorphisms.  It  (arguably) simplifies the task of checking whether $\mathcal{M}$ is cofibrantly generated, by allowing one to completely avoid dealing with transfinite constructions.  Theorem \ref{thm_NewChar} says that cofibrant generation of $\mathcal{M}$ is equivalent to $\mathcal{M}$ being ``almost everywhere effective" (Definition \ref{def_EffectiveSubobj}).  Almost-everywhere effectiveness of $\mathcal{M}$ means that \textbf{if} $\mathfrak{N}$ is a sufficiently elementary submodel of the universe of sets that has access to enough parameters, and \textbf{if} $f: A \to B$ is a member of $\mathfrak{N} \cap \mathcal{M}$, \textbf{then} when we consider the particular diagram below (with lower left part a pushout), the maps $f \restriction \mathfrak{N}$ and $r^{f,\mathfrak{N}}$ must also lie in $\mathcal{M}$:

\[\begin{tikzcd}
	A &&&&& B \\
	\\
	&& {P^{f,\mathfrak{N}}} \\
	\\
	{\mathfrak{N} \cap A} &&&&& {\mathfrak{N} \cap B}
	\arrow["{f \ (\in \mathfrak{N} \cap \mathcal{M})}", hook, from=1-1, to=1-6]
	\arrow[curve={height=12pt}, from=1-1, to=3-3]
	\arrow["{r^{f,\mathfrak{N}}}"', from=3-3, to=1-6]
	\arrow["\ulcorner"{anchor=center, pos=0.125, rotate=-90}, draw=none, from=3-3, to=5-1]
	\arrow["\subseteq", hook, from=5-1, to=1-1]
	\arrow["{f \restriction \mathfrak{N}}"', hook, from=5-1, to=5-6]
	\arrow["\subseteq"', hook, from=5-6, to=1-6]
	\arrow[curve={height=-12pt}, from=5-6, to=3-3]
\end{tikzcd}\]

\noindent This is a strict weakening of a notion due to Barr and Borceux-Rosick\'y, see Definition \ref{def_MainDiag} and subsequent discussion for details.  

The top-down characterization of cofibrant generation in Theorem \ref{thm_NewChar} is also in the background of the main results of \cite{Cox_et_al_CofGenPureMon}, see Remark \ref{rem_bound_loc_cyc}.

\section{Monoid acts and covers} Our terminology agrees with the standard reference for monoid acts (Kilp et al.~\cite{MR1751666}).  A \textbf{monoid} is a nonempty set $S$ together with an associative binary operation that has an identity.  A nonempty subset $K$ of $S$ is a \textbf{left ideal} if $SK \subset K$.  A monoid $S$ is \textbf{right-reversible} if every pair of left ideals of $S$ has nonempty intersection.  For example, every commutative monoid is right-reversible.

Our official definition of a (right) \textbf{$\boldsymbol{S}$-act} is a \emph{non-empty} set $A$ together with a function $(a,s) \mapsto as$ from $A \times S$ to $A$ such that $(as)t = a(st)$ for all $a \in A$ and all $s,t \in S$.  A homomorphism $\varphi: A \to B$ of $S$-acts is a function such that $\varphi(as)=\varphi(a)s $ for all $a \in A$ and $s \in S$.  This category is called \textbf{Act-$\boldsymbol{S}$} in \cite{MR1751666} and appears to be the most commonly-used definition in the act literature.  Because Act-$S$ lacks an initial object, it will sometimes be convenient to work with the larger category $\overline{\textbf{Act}}$-$\boldsymbol{S}$ that includes the empty act (together with a unique morphism $\emptyset \to A$ for each $A \in \{ \emptyset \} \cup \text{Obj}(\text{Act-S})$).  Finally, if $S$ has a left zero $z$ (i.e., $zs=z$ for all $s \in S$) then Act$_0$-$S$ denotes those members of Act-S that have a unique fixed point.  These categories and others are discussed in \cite{MR1751666}, I.6.

If $A$ is a (possibly empty) subact of $B$, the \textbf{Rees congruence} on $B$ is given by $b_0 \rho_A b_1$ if either $b_0 = b_1$ or both $b_0$, $b_1$ are elements of $A$; $B/A$ denotes the quotient act of $B$ by this congruence.  Note that if $A=\emptyset$ then $B/A \simeq B$.  Tensor products of acts is defined similarly as for modules.  An act is \textbf{flat} if tensoring with it preserves monomorphisms.  An act is \textbf{strongly flat} if it is a directed colimit of finitely presented free acts.  Strong flatness is strictly stronger than flatness in Act-$S$ (though they agree in $R$-Mod), and strong flatness is equivalent, by Stenstr\"om~\cite{MR0296191} (using different terminology), to a conjunction of two combinatorial conditions known as conditions (P) and (E); see \cite{MR1751666}, III.9.  $\mathcal{F}_S$ and $\mathcal{SF}_S$ denote the classes of (nonempty) flat and strongly flat $S$-acts, respectively; we omit the $S$ subscript when it is clear from the context.  For a class $\mathcal{X}$ of $S$-acts and an act $A$, an \textbf{$\boldsymbol{\mathcal{X}}$-precover} is a morphism $\pi: X \to A$ with $X \in \mathcal{X}$, such whenever $X' \in \mathcal{X}$ and $\psi: X' \to A$ is a morphism, then $\psi$ factors through $\pi$ (not necessarily uniquely).  If, additionally, the only maps $j: X \to X$ such that $\pi_X  \circ j = \pi_X$ are automorphisms of $X$, then $\pi_X$ is called a $\mathcal{X}$\textbf{-cover}.  

\begin{lemma}[ \cite{MR3251751}, Theorem 4.11 ]\label{lem_PrecoverYieldsCover}
If every $S$-act has an $\mathcal{F}$-precover, then every act has an $\mathcal{F}$-cover, and similarly for $\mathcal{SF}$.  
\end{lemma}

The definitions of cover and precover above are those of Enochs.  A different definition of ``$\mathcal{X}$-cover" was introduced by Mahmoudi and Renshaw, which was equivalent to Enochs' definition in the case $\mathcal{X}$=projectives, but not in general.  See \cite{MR3547004} for a comparison of the two notions.

\section{A new characterization of cellular generation}

\subsection{Cofibrant and cellular generation} 
For morphisms $f$, $g$ in any of the categories discussed above, we say $\boldsymbol{g}$ \textbf{is a pushout of} $\boldsymbol{f}$ if there exists a pushout square of the following form:
\[\begin{tikzcd}
	\bullet & \bullet \\
	\bullet & \bullet
	\arrow["f", from=1-1, to=1-2]
	\arrow[from=1-1, to=2-1]
	\arrow[from=1-2, to=2-2]
	\arrow["g"', from=2-1, to=2-2]
	\arrow["\ulcorner"{anchor=center, pos=0.125, rotate=180}, draw=none, from=2-2, to=1-1]
\end{tikzcd}\]

A class $\mathcal{M}$ of morphisms is \textbf{closed under pushouts} if pushouts of members of $\mathcal{M}$ also lie in $\mathcal{M}$.  A commutative system $\vec{f} = \langle f_{\alpha,\beta}: A_\alpha \to A_\beta \ : \ \alpha \le \beta \le \mu \rangle$ of monoid act morphisms indexed by an ordinal $\mu$ is a \textbf{transfinite composition} if $f_{\alpha,\alpha} = \text{id}_{A_\alpha}$ for each $\alpha$ and, for each limit ordinal $\gamma \le \mu$, $\langle f_{\alpha,\gamma}: A_\alpha \to A_\gamma \ : \ \alpha < \gamma \rangle$ is the colimit of the system $\vec{f} \restriction \gamma$.  ``$f$ is a transfinite composition" is often shorthand for saying there is a transfinite composition $\vec{f}$ with $f =f_{0,\text{lh}(\vec{f})}$.  For a class $\mathcal{M}$ of act morphisms, a \textbf{transfinite composition of members of $\boldsymbol{\mathcal{M}}$} is a transfinite composition $\vec{f}$ such that for each $\alpha < \mu$, the morphism $f_{\alpha,\alpha+1}$ is a member of $\mathcal{M}$.  $\textbf{cell}\boldsymbol{\left( \mathcal{M} \right)}$ denotes the closure of $\mathcal{M}$ under pushouts and transfinite compositions; this is the same as all transfinite compositions of pushouts of members of $\mathcal{M}$ (see \cite{MR3188863}). We will say that $\mathcal{M}$ is \textbf{cellularly generated} if $\mathcal{M}= \text{cell}(\mathcal{M}_0)$ for some \textbf{set} $\mathcal{M}_0$ of morphisms.  $\mathcal{M}$ is \textbf{cofibrantly generated} if it is the closure of a \textbf{set} $\mathcal{M}_0$ under pushouts, transfinite compositions, and retracts. 

\begin{fact}[\cite{MR3188863}]\label{fact_CharCofClos}
In locally presentable categories, if $\mathcal{M}$ is a class of morphisms then
\[
\text{cell} (\mathcal{M}) = \text{Tc} \ \text{Po} \ (\mathcal{M}),
\]
where Po denotes closure under pushouts and Tc denotes closure under transfinite compositions.  The cofibrant closure of $\mathcal{M}$ (smallest class containing $\mathcal{M}$ closed under pushouts, transfinite compositions, and retracts) is equal to
\[
\text{Rt} \ \text{cell} (\mathcal{M})
\]
where Rt denotes closure under retracts.
\end{fact}

\begin{remark}\label{rem_Cell_Cof_equiv}
Since our main Theorem \ref{thm_NewChar} works in the more general situation of non-retract-closed classes, and since there are important examples of non-retract-closed cellularly generated classes (e.g., monomorphisms with free Rees quotient), we speak mainly of cellular generation from now on.
\end{remark}

Cellular generation allows the use of Quillen's \emph{Small Object Argument} originating in homotopy theory.  The following argument, tracing back at least to Rosick\'y~\cite{MR1925005}, explains how cofibrant generation of $\mathcal{X}$-Mono (or $\mathcal{X}$-PureMono) can sometimes yield $\mathcal{X}$-covers.

\begin{lemma}\label{lem_CellGenExpandedGivesCovers}
Suppose $\mathcal{X}$ is a class of (nonempty) $S$-acts closed under disjoint unions and directed colimits, such that for every $A \in$ Act-$S$ there is an $X \in \mathcal{X}$ with $\text{Hom}(X,A) \ne \emptyset$.  Suppose $\mathcal{M}$ is either the class of all monomorphisms, or the class of all pure monomorphisms, in $\overline{\text{Act}}$-$S$.  Let
\[
\mathcal{X}\text{-}\mathcal{M}:= \left\{ f:A \to B \ : \ f \in \mathcal{M} \text{ and } B/\text{im}(f) \in \mathcal{X} \right\}
\]

If $\mathcal{X}$-$\mathcal{M}$ is cellularly generated in $\overline{\text{Act}}$-$S$, then every member of Act-$S$ has an $\mathcal{X}$-cover.
\end{lemma}
\begin{proof}
Note that $\mathcal{M}$ itself is cellularly (in fact cofibrantly) closed in $\overline{\text{Act}}$-$S$.  And if $X \in \mathcal{X}$ then $\emptyset \to X$ is a member of $\mathcal{X}$-$\mathcal{M}$.  Since $\mathcal{X}$-$\mathcal{M}$ is cellularly generated in $\overline{\text{Act}}$-$S$, Quillen's Small Object Argument implies, in particular, that for any nonempty act $A$, the morphism $\emptyset \to A$ factors as 
\[\begin{tikzcd}
	\emptyset &&&& A \\
	&& Q
	\arrow[from=1-1, to=1-5]
	\arrow["f"', from=1-1, to=2-3]
	\arrow["g"', from=2-3, to=1-5]
\end{tikzcd}\]
such that in the category $\overline{\text{Act}}$-$S$:
\begin{itemize}
 \item $f$ is a transfinite composition of pushouts of (the generating set of) members of $\mathcal{X}$-$\mathcal{M}$;
 \item $g$ has the right lifting property in against all members of $\mathcal{X}$-$\mathcal{M}$.
\end{itemize}

We claim that $Q \in \mathcal{X}$ and that $g$ is a $\mathcal{X}$-precover of $A$.  First we prove $Q \ne \emptyset$.  By assumption, there is some $X \in \mathcal{X}$ with $\text{Hom}(X,A)$ nonempty.  Then $\emptyset \to X$ is a member of $\mathcal{X}$-$\mathcal{M}$, and so by the right lifting property of $g$, there is a dashed map as below making the diagram commute:
\[\begin{tikzcd}
	Q & A & {} \\
	\emptyset & X
	\arrow["g", from=1-1, to=1-2]
	\arrow[from=2-1, to=1-1]
	\arrow[from=2-1, to=2-2]
	\arrow[dashed, from=2-2, to=1-1]
	\arrow["\pi"', from=2-2, to=1-2]
\end{tikzcd}\]
Since $X \ne \emptyset$ (because $\mathcal{X}$ consists of nonempty acts by assumption), this implies $Q \ne \emptyset$.  A similar argument shows that, if we can show that $Q \in \mathcal{X}$, it will follow that $g$ is an $\mathcal{X}$-precover of $A$ in Act-$S$.

So it remains to prove that $Q \in \mathcal{X}$.  Now $f$ is a transfinite composition $f = f_{0,\mu}: Q_0=\emptyset \to Q=Q_\mu$, where for every $\alpha < \mu$, $f_{\alpha,\alpha+1}: Q_\alpha \to Q_{\alpha+1}$ is a pushout of \emph{either} a monomorphism in $\mathcal{M}$ of the form $\emptyset \to X$ for some $X \in \mathcal{X}$, \emph{or} a monomorphism in $\mathcal{M}$ of the form $\emptyset \ne B \to C$ with $C/B \in \mathcal{X}$.  Since $Q \ne \emptyset$, we can without loss of generality assume $Q_1 \ne \emptyset$; hence $\emptyset=Q_0 \to Q_1$ must have been a trivial pushout of the former kind, with $Q_1/\emptyset = Q_1 \in \mathcal{X}$.  We prove by induction on $\alpha$ that each $Q_\alpha \in \mathcal{X}$ for $\alpha \ge 1$.  The base case $\alpha = 1$ was just shown, and for limit $\alpha$ we use the assumption that $\mathcal{X}$ is closed under directed colimits and that $\mathcal{M}$ is cellularly closed.  So it remains to show that if $Q_\alpha \in \mathcal{X}$ then so is $Q_{\alpha+1}$.   If $Q_\alpha \to Q_{\alpha+1}$ is a pushout of some $\emptyset \to X$ with $X \in \mathcal{X}$, then $Q_{\alpha+1} \simeq Q_\alpha \sqcup X$, which is a member of $\mathcal{X}$ by closure under disjoint unions.  If $Q_\alpha \to Q_{\alpha+1}$ is a pushout of some $\emptyset \ne B \to C$ in $\mathcal{M}$ with $C/B \in \mathcal{X}$, then since $\mathcal{M}$ is closed under pushouts, and pushouts preserve Rees quotients of nonempty acts (\cite[Theorem 3.11 part 1]{MR3249868}), $Q_\alpha \to Q_{\alpha+1}$ is in $\mathcal{X}$-$\mathcal{M}$.  And since $Q_\alpha \in \mathcal{X}$, $\emptyset \to Q_\alpha$ is also in $\mathcal{X}$-$\mathcal{M}$.  Since $\mathcal{X}$-$\mathcal{M}$ is assumed to be cellularly closed in $\overline{\text{Act}}$-$S$, then in particular it is closed under 2-step composition, so $\emptyset \to Q_{\alpha+1}$ is in $\mathcal{X}$-$\mathcal{M}$; so $Q_{\alpha+1}/\emptyset \simeq Q_{\alpha+1} \in \mathcal{X}$.

We have shown every act in Act-$S$ has a $\mathcal{X}$-precover. Since $\mathcal{X}$ is closed under directed colimits, Lemma \ref{lem_PrecoverYieldsCover} implies every act has a $\mathcal{X}$-cover.
\end{proof}

\begin{corollary}
If either $\mathcal{F}$-Mono or $\mathcal{F}$-PureMono is cellularly generated in $\overline{\text{Act}}$-$S$, then every nonempty act has a $\mathcal{F}$-cover.  The same is true if we replace $\mathcal{F}$ with $\mathcal{SF}$.
\end{corollary}

\subsection{Set theory preliminaries}

We will use set-theoretic methods to prove the main results discussed in the introduction.  We work in the Zermelo Fraenkel axioms with Choice (ZFC).  $V$ denotes the universe of sets and $\in$ denotes the membership relation.   We refer the reader to Jech~\cite{MR1940513} for an overview of the \emph{satisfaction relation} $\models$ and \emph{elementary substructure relation} $\prec$ from first order logic, and how they are often used specifically in the context of set theory.

 Given a regular uncountable cardinal $\kappa$ and a set $X$ of size $<\kappa$, we would often \emph{like} to find an elementary submodel of the universe $(V,\in)$ of size $<\kappa$ that contains $X$ as a subset; i.e., we would like to have access to a set $N$ of cardinality $<\kappa$ such that
\[
X \subset N \text{ and } \mathfrak{N}:= (N,\in) \prec (V,\in),
\]
where $(N,\in)$ denotes the structure with universe $N$ and predicate 
\[
\in \cap (N \times N) = \{ (a,b) \ : \ a,b \in N \text{ and } a \in b \}.
\]

Due to constraints imposed by G\"{o}del's Second Incompleteness Theorem, we cannot literally do this.\footnote{Though we could, if one is willing to assume a very mild large cardinal axiom, such as ``ORD is Mahlo".}  However, this issue can almost always be circumvented in practice, and various methods are discussed extensively in various excellent articles on the use of set-theoretic elementary submodels (Dow~\cite{MR1031969}, Geschke~\cite{MR1950041}, Just-Weese~\cite{MR1474727}, Soukup~\cite{MR2800978}).  These technical issues also come up in the theory of forcing (Chapter VII of Kunen~\cite{MR597342} has a good summary).  

In order to  prevent these technical issues from obscuring the main arguments of the paper, and to avoid having to repeatedly perform a tedious analysis of formula complexity in the language of set theory,\footnote{The language whose first order signature has a single binary predicate symbol $\dot{\in}$.}
 we choose to essentially mimic Kunen's~\cite{MR597342} approach.  For a (meta-mathematical) natural number $n$, we consider the Levy class $\Sigma_n$ of formulas in the language of set theory.  The definition of the hierarchy isn't needed here, just that
\begin{enumerate*}
 \item every formula is a $\Sigma_n$ formula for some $n$;
 \item if $m \le n$, every $\Sigma_m$ formula is also $\Sigma_n$; and
 \item for fixed $n$, the $\Sigma_n$ satisfaction relation is first-order definable (see Kanamori~\cite{MR1994835}, chapter 0); and
 \item For fixed $n$, there is a closed unbounded class $C_n$ of ordinals $\alpha$ such that $(V_\alpha,\in) \prec_{n} (V,\in)$ for every $\alpha \in C_n$, where $\prec_n$ denotes elementarity with respect to all $\Sigma_n$ formulas in the language of set theory.
\end{enumerate*}  
The latter fact is due to the first-order expressibility of $\models_{n}$ and the Levy-Montague Reflection Principle.  We note that in general, the classes $C_n$ are \textbf{not} uniformly defined across $n$.

In particular, for any fixed $n$, any regular uncountable cardinal $\kappa$, and any set $X$ of cardinality $<\kappa$, there is an $\alpha \in C_n$ with $X \subset V_\alpha$ and $(V_\alpha,\in) \prec_n (V,\in)$.  Then one can use the ordinary Downward L\"owenheim-Skolem Theorem on the set-sized structure $(V_\alpha,\in)$ to find an $N$ with $|N|<\kappa$, $X \subset N$, and $(N,\in) \prec (V_\alpha,\in)$ (and $N \cap \kappa \in \kappa$, which will be convenient).  It follows that $(N,\in) \prec_n (V,\in)$.    Hence there is a scheme:

\begin{fact}[Class $\Sigma_n$ L\"owenheim-Skolem Theorem]\label{fact_LS_class}
For any (meta-mathematical) natural number $n$ and any regular uncountable cardinal $\kappa$:
for every set $X$ of size $<\kappa$, there exists some $N$ of size $<\kappa$ such that $X \subseteq N$, $N \cap \kappa \in \kappa$, and
\[
(N,\in) \prec_n (V,\in),
\]
which means that for every $\Sigma_n$ formula $\varphi(v_0, \dots, v_k)$ and every $a_0,\dots, a_k$ in $N$, 
\[
(N,\in) \models \varphi[a_0,\dots,a_k] \text{ if and only if } (V,\in) \models \varphi[a_0,\dots,a_k].
\]
\end{fact}

We will often use $\mathfrak{N}$ to denote the structure $(N,\in)$, and write $X \subset \mathfrak{N}$ when we really mean $X \subset N$.  

\begin{convention}
When we say \textbf{``$\mathfrak{N}$ is a sufficiently elementary submodel of $\boldsymbol{(V,\in)}$"}, or write 
\[
\mathfrak{N} \boldsymbol{\prec_*} (V,\in),
\]
we mean that $\mathfrak{N} \prec_n (V,\in)$ for some large enough $n$ for the purpose at hand.
\end{convention}

For example, one could just let $n$ be an upper bound on the complexities of all the (finitely many) set-theoretic statements appearing in this paper.  It turns out that all set-theoretic statements used in elementarity arguments in this paper can be written in a $\Sigma_1$ manner (so $n=1$ would work), but to prove that would be tedious, unnecessary, and obscure the main point.\footnote{For example, in the proofs we often want to reflect set-theoretic statements like ``$D$ is the directed colimit of the system $\mathcal{S}$"; this statement can be expressed in a $\Sigma_1$ way if the category is (say) locally finitely presentable, but such computations can be tedious.  There are contexts where it is crucial for checking that some statement is $\Sigma_1$-expressible (e.g., for some of the results in Bagaria et al.~\cite{MR3323199} or Cox~\cite{Cox_MaxDecon} where one not only needs elementarity of $\mathfrak{N}$ but also of its Mostowski collapse).}

Earlier we mentioned that the ``$\mathfrak{N} \cap \kappa \in \kappa$" requirement of Fact \ref{fact_LS_class} is for convenience; the next fact shows why.  To say that $N \cap \kappa$ is transitive just means it is downward closed as a set of ordinals; i.e., either $N \cap \kappa = \kappa$ or $N \cap \kappa \in \kappa$.
\begin{fact}\label{fact_element_subset}
If $\kappa$ is regular and uncountable and $N$ such that $\mathfrak{N} = (N,\in) \prec_* (V,\in)$ and $N \cap \kappa$ is transitive, the following holds:  
 \[
 \forall z \ \left( z \in N \text{ and } |z|<\kappa \right) \implies z \subset N.
 \]

\end{fact}
The fact is well-known but we provide the short proof:  since $z \in \mathfrak{N} \prec_* (V,\in)$, the cardinality $|z|$ (least ordinal of the same cardinality as $z$) is also an element of $\mathfrak{N}$. By assumption that $|z| < \kappa$ and $\mathfrak{N} \cap \kappa$ is transitive, it follows that $|z| \subset \mathfrak{N}$.  By elementarity there is a surjection $f: |z| \to z$ with $f \in \mathfrak{N}$.  Since $|z|=\text{dom}(f) \subset \mathfrak{N}$, the outputs of $f$ all lie in $\mathfrak{N}$ too, so $z \subset \mathfrak{N}$.

We describe how elementarity of $\mathfrak{N}$ in $(V,\in)$ can be useful in the context of monoid acts.  Suppose $S$ is a monoid and $S \cup \{ S \} \subset \mathfrak{N} \prec_* (V,\in)$; note that $S$ is both an element, and subset, of $\mathfrak{N}$.  To say that $S \in \mathfrak{N}$ is of course shorthand for saying that the underlying set of the monoid, together with the monoid operation $\cdot: S \times S \to S$, are elements of $\mathfrak{N}$.  Suppose $A$ is an $S$-act with $A \in \mathfrak{N}$ (which, similarly, is shorthand for saying that $A \in \mathfrak{N}$ and the action $\psi_A: A \times S \to A$ that defines the action of $S$ on $A$, is an element of $\mathfrak{N}$).  Then $\mathfrak{N} \cap A$ is a subact of $A$.  To see this, consider an arbitrary $a \in \mathfrak{N} \cap A$ and an arbitrary $s \in S$; then $s$ is also an element of $\mathfrak{N}$, by the assumption that $S \subset \mathfrak{N}$.  Now the universe $(V,\in)$ satisfies the formula
\begin{equation}\label{eq_SampleForm}
\exists  b \in A \ \ \psi_A(a,s) =b,
\end{equation}
whose free parameters are $A$, $\psi_A$, $a$, and $s$.\footnote{The displayed statement can of course be fully rewritten in the language of set theory, but it is tedious even to express ordered pairs.  See \cite{MR1940513} for such commonly-used translations of ordinary math into the ``machine language" of set theory.}  Because those free parameters are all elements of $\mathfrak{N}$ and $\mathfrak{N} \prec_* (V,\in)$, it follows that $\mathfrak{N}$ also satisfies the formula \eqref{eq_SampleForm}.  Then by definition of the satisfaction relation, this means there is a $b \in N \cap A$ such that $\mathfrak{N}$ satisfies ``$\psi_A(a,s)=b$", i.e., $\mathfrak{N}$ satisfies ``$\big( (a,s),b \big) \in \psi_A$".  This is essentially an atomic formula,\footnote{Leaving aside issues about coding ordered tuples.} and since the predicate on $\mathfrak{N}$ is $\in \cap (N \times N)$, this implies $\psi_A(a,s)=b$.  So $\mathfrak{N} \cap A$ is closed under the action.  A similar argument shows that if $A$ has a unique fixed point, then it must be a member of $\mathfrak{N} \cap A$; so
\[
A \in \text{Act}_0\text{-}S \ \implies \ \mathfrak{N} \cap A \in \text{Act}_0\text{-}S.
\]  

The next lemma is crucial to the proofs of the main theorems.
\begin{lemma}\label{lem_HomThm}
Suppose $S$ is a monoid and $S \cup \{S \} \subset \mathfrak{N} \prec_* (V,\in)$.  
\begin{enumerate}[label=(\roman*)]
 \item\label{item_RestrictCongQuot} If $B$ is an $S$-act and $\rho$ is a congruence on $B$, with $B$ and $\rho$ both elements of $\mathfrak{N}$, then $\mathfrak{N} \cap \rho$ is a congruence on $\mathfrak{N} \cap B$, and $\mathfrak{N} \cap \frac{B}{\rho} \simeq \frac{\mathfrak{N} \cap B}{\mathfrak{N} \cap \rho}$.  In particular, if $\rho$ is a Rees congruence with respect to a subact $A$ of $B$, then $\mathfrak{N} \cap \frac{B}{A} \simeq \frac{\mathfrak{N} \cap B}{\mathfrak{N} \cap A}$.
 \item\label{item_QuotPartOverN} If $A$ is a subact of $B$ with $A$ and $B$ both elements of $\mathfrak{N}$, then $B/A \in \mathfrak{N}$ and there is an isomorphism between the following Rees quotients: 
\[ 
\frac{B}{A \cup (\mathfrak{N}\cap B)} \simeq \frac{B/A}{\mathfrak{N} \cap (B/A)}.
\]
 
\end{enumerate}

\end{lemma}
\begin{proof}
Note that the assumption $S \cup \{ S \} \subset \mathfrak{N} \prec_* (V,\in)$ implies that $\mathfrak{N} \cap A$ is a subact of $A$, whenever $A$ is an $S$-act with $A \in \mathfrak{N}$ (by the discussion above).

\ref{item_RestrictCongQuot}:  That $\mathfrak{N} \cap \rho$ is an equivalence relation on $\mathfrak{N} \cap B$, and that $\mathfrak{N} \cap \rho$ is invariant w.r.t.\ multiplication by $S$, are straightforward applications of the assumptions that $S \cup \{S \} \subset \mathfrak{N} \prec_* (V,\in)$ and $\rho \in \mathfrak{N}$.  To see that $\mathfrak{N} \cap \frac{B}{\rho} \simeq \frac{\mathfrak{N} \cap B}{\mathfrak{N} \cap \rho}$, define $\varphi: \mathfrak{N} \cap B \to \mathfrak{N} \cap \frac{B}{\rho}$ by sending any $b \in \mathfrak{N} \cap B$ to $[b]_\rho$.  This is clearly an act homomorphism, and it is surjective by elementarity of $\mathfrak{N}$, since for any equivalence class $X \in \mathfrak{N} \cap \frac{B}{\rho}$, $\mathfrak{N}$ sees some $b \in B$ such that $[b]_\rho=X$.  And for any $b_0,b_1 \in \mathfrak{N} \cap B$,
\begin{equation*}
(b_0,b_1) \in \text{ker} \varphi \iff  \varphi(b_0) = \varphi(b_1) \iff [b_0]_{\rho} = [b_1]_\rho \iff (b_0,b_1) \in \mathfrak{N} \cap \rho.
\end{equation*}
Then by the Homomorphism Theorem (Chapter I Corollary 4.22 of \cite{MR1751666}), 
\[
\frac{\mathfrak{N} \cap B}{\mathfrak{N} \cap \rho} = \frac{\mathfrak{N} \cap B}{\text{ker}\varphi} \simeq \text{im} \varphi = \mathfrak{N} \cap \frac{B}{\rho}.
\]

\ref{item_QuotPartOverN}: Define
\[
\psi: B \to \frac{B/A}{\mathfrak{N} \cap (B/A)}, \ \ b \mapsto \big[ [b]_A \big]_{\mathfrak{N} \cap (B/A)}.
\]
Clearly this is a surjective homomorphism.  We claim that the congruence $\text{ker} \psi$ is the same as the Rees congruence on $B$ induced by the subact $A \cup (\mathfrak{N} \cap B)$.  For any $b,b' \in B$:
\begin{align}
& (b_0,b_1) \in \text{ker} \psi  & \\
\iff & \psi(b_0)=\psi(b_1) & \\
\iff & \big[ [b_0]_A \big]_{\mathfrak{N} \cap (B/A)}=\big[ [b_1]_A \big]_{\mathfrak{N} \cap (B/A)} & \\
\iff & [b_0]_A = [b_1]_A \text{ or } \Big(  [b_0]_A, \  [b_1]_A \in  \mathfrak{N} \Big) & \\
\iff & b_0=b_1 \text{ or } \Big( b_0,b_1 \in A \cup (\mathfrak{N} \cap B)  \Big)&   
\end{align}
The only nontrivial equivalence is the last one.  To see the $\Rightarrow$ direction:  if $[b_0]_A = [b_1]_A$ then either $b_0 = b_1$ or both $b_0$ and $b_1$ are in $A$, so the last line holds.  Now assume both $X:=[b_0]_A$ and $Y:=[b_1]_A$ are elements of $\mathfrak{N}$ (note here we are \emph{not} assuming $b_0$ and $b_1$ are elements of $\mathfrak{N}$; just that their equivalence classes are).  We show both $b_0$ and $b_1$ must lie in $A \cup (\mathfrak{N} \cap B)$.  If not, then at least one of them, say, $b_0$, must lie in $B \setminus (A \cup  \mathfrak{N})$.  Since $b_0 \notin A$ then $X=[b_0]_A = \{ b_0 \}$ and hence, since $X \in \mathfrak{N} \prec_* (V,\in)$, the unique element of $X$, i.e., $b_0$, must lie in $\mathfrak{N}$, which is a contradiction.  

To see the $\Leftarrow$ direction of the final equivalence, assume $b_0, b_1$ are both in $A \cup (\mathfrak{N} \cap B)$.  Then for each $i$, either $[b_i]_A = A$ which is an element of $\mathfrak{N}$ (since $A \in \mathfrak{N}$ by assumption), or else $b_i \in \mathfrak{N}$ and hence (by elementarity of $\mathfrak{N}$) $[b_i]_A \in \mathfrak{N}$.  So both $[b_0]_A$ and $[b_1]_A$ lie in $\mathfrak{N}$.
\end{proof}

\subsection{The new characterization}

Suppose $f: A \to B$ is a morphism of $S$-acts with $f \in \mathfrak{N}$ (by which we mean $f$, $A$, and $B$ are elements of $\mathfrak{N}$) and $S \cup \{S \} \subset \mathfrak{N} \prec_* (V,\in)$.  Then since $f \in \mathfrak{N} \prec_* (V,\in)$ it follows that $f(a) \in \mathfrak{N}$ whenever $a \in \mathfrak{N} \cap A$.  Hence, the restriction of $f$ to $\mathfrak{N} \cap A$ maps \emph{into} $\mathfrak{N} \cap B$.  We let $f \restriction \mathfrak{N}$ denote this restriction, with domain $\mathfrak{N} \cap A$ and codomain $\mathfrak{N} \cap B$.  Hence, the bottom row of the diagram in the following definition makes sense, and the outer vertical maps are embeddings of $S$-acts:

\begin{definition}\label{def_MainDiag}
Suppose $S$ is a monoid and $\mathcal{K}$ is one of $\overline{\text{Act}}$-$S$, Act-$S$, or Act$_0$-$S$.  Suppose $S \cup \{ S \} \subset \mathfrak{N} \prec_* (V,\in)$.  For any morphism $f: A \to B$ such that $f \in \mathfrak{N}$, $\mathcal{D}^{f,\mathfrak{N}}$ denotes the following commutative diagram, whose lower left portion is a pushout, and where $r^{f,\mathfrak{N}}$ denotes the unique map from the pushout that commutes with the rest of the diagram:

\[\begin{tikzcd}
	& A &&&&& B \\
	\\
	{\mathcal{D}^{f,\mathfrak{N}}:=} &&& {P^{f,\mathfrak{N}}} \\
	\\
	& {\mathfrak{N} \cap A} &&&&& {\mathfrak{N} \cap B}
	\arrow["{f \ (\in \mathfrak{N})}", hook, from=1-2, to=1-7]
	\arrow[curve={height=12pt}, from=1-2, to=3-4]
	\arrow["{r^{f,\mathfrak{N}}}"', from=3-4, to=1-7]
	\arrow["\ulcorner"{anchor=center, pos=0.125, rotate=-90}, draw=none, from=3-4, to=5-2]
	\arrow["\subseteq", hook, from=5-2, to=1-2]
	\arrow["{f \restriction \mathfrak{N}}"', hook, from=5-2, to=5-7]
	\arrow["\subseteq"', hook, from=5-7, to=1-7]
	\arrow[curve={height=-12pt}, from=5-7, to=3-4]
\end{tikzcd}\]
\end{definition}

Motivated by work of Barr, Borceaux-Rosick\'y~\cite{MR2280436} defined what it means for a locally finitely presentable category to satisfy ``effective unions of $\mathcal{M}$-subobjects" for a class $\mathcal{M}$ of morphisms.  This was a sufficient, but not necessary, condition for $\mathcal{M}$ to be cellularly generated.  It is not necessary because, for example, the class of pure monomorphisms is cofibrantly generated in the category \textbf{Ab} of abelian groups, but does not have effective unions of subobjects in \textbf{Ab} (\cite{MR4121092}).  We realized that if one weakens their condition to merely hold ``almost everywhere", then this exactly characterizes cellular generation.

\begin{definition}\label{def_EffectiveSubobj}
Suppose $S$ is a monoid, $\mathcal{K}$ is either Act-$S$, $\overline{\text{Act}}$-$S$, or Act$_0$-$S$, and $\mathcal{M}$ is a class of morphisms in $\mathcal{K}$ that is cellularly closed (i.e., closed under pushouts and transfinite compositions).  For a regular cardinal $\kappa > |S|$, we say that \textbf{$\boldsymbol{\mathcal{M}}$ is $\boldsymbol{\kappa}$-almost-everywhere (a.e.) effective} if the following holds for some natural number $n$:\footnote{This definition is really a definition \emph{scheme} indexed by natural numbers.}
\begin{align*}
\left( \exists p  \  \forall \mathfrak{N} \right) & \Big( \textbf{if } S \cup \{S,p \} \subset \mathfrak{N}\prec_{\Sigma_n} (V,\in) \text{ and } \mathfrak{N} \cap \kappa \text{ is transitive,} \\
& \textbf{then } (\forall f \in \mathfrak{N} \cap \mathcal{M})( f \restriction \mathfrak{N} \text{ and } r^{f,\mathfrak{N}} \text{ are both in } \mathcal{M} )\Big),
\end{align*}
where $r^{f,\mathfrak{N}}$ denotes the map from Definition \ref{def_MainDiag}.  We say \textbf{$\boldsymbol{\mathcal{M}}$ is a.e.-effective} if it is $\kappa$-a.e.\ effective for some regular $\kappa$ (equivalently, a tail end of regular $\kappa$, see below).   
\end{definition}
\noindent The ``almost everywhere" terminology is chosen because one may rephrase the displayed statement using a slightly nonstandard interpretation of Shelah's ``almost all" quantifier:
\[
\left( \text{aa}^*_\kappa \mathfrak{N} \right) \left( \forall f \in \mathfrak{N} \cap \mathcal{M} \right)\left(  f \restriction \mathfrak{N} \in \mathcal{M} \text{ and } r^{f,\mathfrak{N}} \in \mathcal{M}\right).
\]
Intuitively, the ``almost" in ``almost all" is because of the presence of the constraint that $\mathfrak{N}$ has to include the parameter $p$ and the parameters in $S \cup \{ S \}$, and be sufficiently elementary in $(V,\in)$.  The $p$ will very often be empty or trivial.  One could also make sense of ``restriction to $\mathfrak{N}$" for many other categories (e.g., any locally presentable category), but we stick with monoid acts for simplicity.  

If $\kappa_0 < \kappa_1$ are regular cardinals and $\mathfrak{N}\cap \kappa_1$ is transitive, then $\mathfrak{N} \cap \kappa_0$ is transitive.  It follows easily that if $\mathcal{M}$ is $\kappa_0$-a.e.\ effective then it is $\kappa_1$-a.e.\ effective.  This justifies the parenthetical comment in the definition.

Before proving our key Theorem \ref{thm_NewChar} we need the next lemma.
  
\begin{lemma}\label{lem_restrictColimit}
Suppose $S$ is a monoid and $\mathcal{K}$ is one of $\overline{\text{Act}}$-$S$, Act-$S$, or Act$_0$-$S$.  Suppose 
\[\begin{tikzcd}
	A && C \\
	\\
	B && P
	\arrow["f", from=1-1, to=1-3]
	\arrow["g"', from=1-1, to=3-1]
	\arrow["{g'}", from=1-3, to=3-3]
	\arrow["{f'}"', from=3-1, to=3-3]
	\arrow["\ulcorner"{anchor=center, pos=0.125, rotate=180}, draw=none, from=3-3, to=1-1]
\end{tikzcd}\]
is a pushout square in $\mathcal{K}$.  If $S \cup \{ S,f,g,f',g' \} \subset \mathfrak{N} \prec_* (V,\in)$, then the square's restriction to $\mathfrak{N}$
\[\begin{tikzcd}
	{\mathfrak{N} \cap A} && {\mathfrak{N} \cap C} \\
	\\
	{\mathfrak{N} \cap B} && {\mathfrak{N} \cap P}
	\arrow["{f \restriction \mathfrak{N}}", from=1-1, to=1-3]
	\arrow["{g \restriction \mathfrak{N}}"', from=1-1, to=3-1]
	\arrow["{g' \restriction \mathfrak{N}}", from=1-3, to=3-3]
	\arrow["{f' \restriction \mathfrak{N}}"', from=3-1, to=3-3]
\end{tikzcd}\]
is also a pushout square.\footnote{This is an instance of a more general fact:  if $\mathcal{D}$ is a diagram in a locally finitely presentable category $\mathcal{K}$, $\mathcal{D} \in \mathfrak{N} \prec_* (V,\in)$, and (a representative set of) the finitely presentable objects of $\mathcal{K}$ is contained in $\mathfrak{N}$, then it is possible to make sense of ``restriction to $\mathfrak{N}$", and one can show $\text{colim} \left( \mathcal{D} \restriction \mathfrak{N} \right) = \left( \text{colim} \mathcal{D} \right) \restriction \mathfrak{N}$.}
\end{lemma}

\begin{proof}
The pushout $P$ can be explicitly described as 
\[
P = \frac{B \sqcup C}{\overline{R_{f,g}}}
\]  
where $\overline{R_{f,g}}$ is the equivalence relation generated by the binary relation
\[
R_{f,g}:=\{ (g(a),f(a))  \ : \ a \in A \}
\]
on $B \sqcup C$.  And $f'$ and $g'$ are the respective compositions of $f$ and $g$ with the factor map $B \sqcup C \to P$.  

By Lemma \ref{lem_HomThm}
\[
\mathfrak{N} \cap P = \mathfrak{N} \cap \frac{B \sqcup C}{\overline{R_{f,g}}} \simeq \frac{\mathfrak{N} \cap (B \sqcup C)}{\mathfrak{N} \cap \overline{R_{f,g}}} 
\]
Now $\mathfrak{N} \cap (B \sqcup C) = (\mathfrak{N} \cap B) \sqcup (\mathfrak{N} \cap C)$, and elementarity of $\mathfrak{N}$ (and the fact that $f,g \in \mathfrak{N}$) imply that $\overline{R_{f,g}}$ is the same as the equivalence relation generated by the relation
\[
R_{f \restriction \mathfrak{N}, g \restriction \mathfrak{N}} = \{ (g(a),f(a)) \ : \ a \in \mathfrak{N} \cap A  \}.
\]
Then
\[
\frac{\mathfrak{N} \cap (B \sqcup C)}{\mathfrak{N} \cap \overline{R_{f,g}}}  = \frac{(\mathfrak{N} \cap B) \sqcup (\mathfrak{N} \cap C)}{\overline{R_{f \restriction{N}, g \restriction \mathfrak{N}}}}
\]
which is the pushout of the span created by $f \restriction \mathfrak{N}$ and $g \restriction \mathfrak{N}$.
\end{proof}

\begin{theorem}\label{thm_NewChar}
Suppose $S$ is a monoid and $\mathcal{K}$ is one of $\overline{\text{Act}}$-$S$, Act-$S$, or Act$_0$-$S$.  Suppose $\mathcal{M}$ is a cellularly-closed class of monomorphisms in $\mathcal{K}$.  The following are equivalent:

\begin{enumerate}[label=(\Roman*)]
 \item\label{item_AE} $\mathcal{M}$ is almost-everywhere effective.
 \item\label{item_CellGen} $\mathcal{M}$ is cellularly generated.
\end{enumerate}

\end{theorem}

\begin{remark}
Theorem \ref{thm_NewChar} holds in more generality.  Namely, the assumption that $\mathcal{M} \subseteq \text{Mono}$ can be dropped, and the proof works for any locally presentable category (under the appropriately modified notion of ``restriction to $\mathfrak{N}$").   This will appear in a separate paper.  We stated Theorem \ref{thm_NewChar} in the form above because it is sufficient for the applications in this paper, and its proof is simpler than the general situation.  
\end{remark}

\begin{proof}
(of Theorem \ref{thm_NewChar}):  We prove the theorem for the category $\overline{\text{Act}}$-$S$; the proofs for Act-$S$ and Act$_0$-$S$ are basically identical, since pushouts in those categories agree with pushouts from the ambient category $\overline{\text{Act}}$-$S$.
 
\ref{item_AE} $\implies$ \ref{item_CellGen}:  Let $\kappa > |S|$ witness that $\mathcal{M}$ is a.e.\ effective.  We prove by induction on cardinals $\lambda$, that if $f \in \mathcal{M}$ and $|\text{cod}(f)| \le \lambda$, then $f \in \text{cell} \big(\mathcal{M}^{<\kappa}\big)$, where $\mathcal{M}^{<\kappa}$ denotes a representative set of members of $\mathcal{M}$ whose codomains have cardinality $<\kappa$.

This trivially holds for $\lambda <\kappa$.  Assume that $\lambda \ge \kappa$ is a cardinal, and that the claim holds for all members of $\mathcal{M}$ with codomains of size $<\lambda$.  Consider an arbitrary $f: A \to B$ in $\mathcal{M}$ with $|B|=\lambda$; since $\mathcal{M} \subset \text{Mono}$ we can without loss of generality assume $f$ as an inclusion.  By a basic application of Fact \ref{fact_LS_class} (see, e.g., Fact 2.2 of \cite{Cox_MaxDecon}), there is a sequence $\langle \mathfrak{N}_\alpha \ : \ \alpha \le \text{cf}(\lambda) \rangle$\footnote{$\text{cf}(\lambda)$ denotes the cofinality of $\lambda$, i.e., the least cardinal $\mu$ such that there is an increasing and cofinal function $\mu \to \lambda$.} of sufficiently elementary submodels of $(V,\in)$\footnote{I.e., a fixed sufficiently large $n$ such that $\mathfrak{N}_\alpha \prec_n (V,\in)$ for each $\alpha$.} such that 
\begin{enumerate}
 \item $\mathfrak{N}_\alpha$ is both an element and subset of $\mathfrak{N}_{\alpha+1}$ for each $\alpha < \text{cf}(\lambda)$;
 \item each $\mathfrak{N}_\alpha$ has transitive intersection with $\kappa$; 
 \item if $\alpha < \text{cf}(\lambda)$ then $|\mathfrak{N}_\alpha|<\lambda$;
 \item $S \cup \{ S,A,B\} \subset \mathfrak{N}_0$;

 \item (continuity) for all limit ordinals $\alpha \le \text{cf}(\lambda)$, $\mathfrak{N}_\alpha = \bigcup_{\zeta < \alpha} \mathfrak{N}_\zeta$; and

 \item $B \subset \mathfrak{N}_{\text{cf}(\lambda)}$.
\end{enumerate}
By assumption \ref{item_AE},

\begin{equation}\label{eq_all_alpha_h_nice}
\forall \alpha \le \text{cf}(\lambda) \ \forall h \in \mathcal{M} \cap \mathfrak{N}_\alpha: \text{ both } h \restriction \mathfrak{N}_\alpha \text{ and } r^{h,\mathfrak{N}_\alpha} \text{ are members of } \mathcal{M}. \tag{*}
\end{equation} 
Note the ``diagonal" nature of \eqref{eq_all_alpha_h_nice}; it quantifies over \emph{all} members $h$ of $\mathcal{M} \cap \mathfrak{N}_\alpha$, not just the inclusion $f: A \subset B$.  It follows immediately (since $f \in \mathfrak{N}_0 \subseteq \mathfrak{N}_\alpha$ for all $\alpha$) that 

\begin{equation}\label{eq_rest_in_M}
\forall \alpha \ \ f \restriction \mathfrak{N}_\alpha \in \mathcal{M} \text{ and } r^{f,\mathfrak{N}_\alpha} \in \mathcal{M}.
\end{equation}

Moreover, if $\alpha < \text{cf}(\lambda)$, then since $\mathfrak{N}_\alpha$ is an \emph{element} of $\mathfrak{N}_{\alpha+1}$, the diagram $\mathcal{D}^{f,\mathfrak{N}_\alpha}$ and in particular the function $r^{f,\mathfrak{N}_\alpha}$ is an \emph{element} of $\mathfrak{N}_{\alpha+1}$, since $\mathfrak{N} \prec_* (V,\in)$. So it makes sense to restrict $r^{f,\mathfrak{N}_\alpha}$ to $\mathfrak{N}_{\alpha+1}$.  Then we apply \eqref{eq_all_alpha_h_nice} again (with $h=r^{f,\mathfrak{N}_\alpha} \in \mathfrak{N}_{\alpha+1}$) to conclude that
 
\begin{equation}\label{eq_quot_in_M}
\forall \alpha < \text{cf}(\lambda)  \ \  r^{f,\mathfrak{N}_\alpha} \restriction \mathfrak{N}_{\alpha+1} \text{ is a member of } \mathcal{M}.
\end{equation}

We will use these facts to write $f$ as a transfinite composition of maps from $\mathcal{M}$ that each have codomain of size $<\lambda$, and then use the induction hypothesis on those pieces to conclude that $f$ is in $\text{cell}(\mathcal{M}^{<\kappa})$.

For each $\alpha$ consider the diagram $\mathcal{D}^{f,\mathfrak{N}_\alpha}$.  Since $A \cap (\mathfrak{N}_\alpha \cap B) = \mathfrak{N}_\alpha \cap A$, the pushout of the span $A \leftarrow \mathfrak{N}_\alpha \cap A \rightarrow \mathfrak{N}_\alpha \cap B$ is exactly $A \cup (\mathfrak{N}_\alpha \cap B)$, and the diagram $\mathcal{D}^{f,\mathfrak{N}_\alpha}$ has this form, with all arrows inclusions (recall our assumption that $f$ itself was an inclusion):
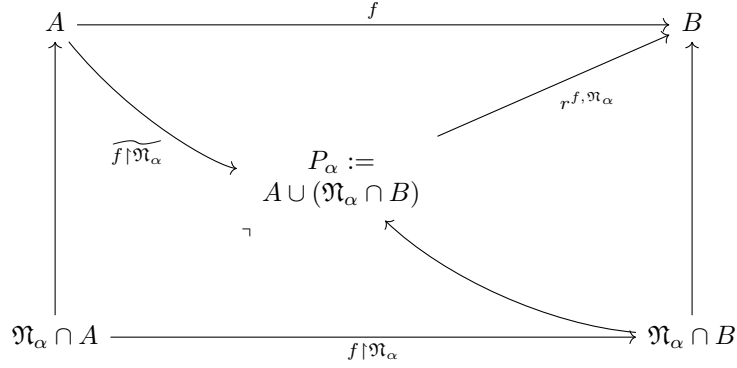
\[\begin{tikzcd}
	A &&&&& B \\
	\\
	&& \begin{array}{c} P_\alpha:= \\ A \cup (\mathfrak{N}_\alpha \cap B) \end{array} \\
	\\
	{\mathfrak{N}_\alpha \cap A} &&&&& {\mathfrak{N}_\alpha \cap B}
	\arrow["f", from=1-1, to=1-6]
	\arrow["{\widetilde{f \restriction \mathfrak{N}_\alpha}}"', curve={height=18pt}, from=1-1, to=3-3]
	\arrow["{r^{f,\mathfrak{N}_\alpha}}"', from=3-3, to=1-6]
	\arrow["\ulcorner"{anchor=center, pos=0.125, rotate=-90}, draw=none, from=3-3, to=5-1]
	\arrow[from=5-1, to=1-1]
	\arrow["{f \restriction \mathfrak{N}_\alpha}"', from=5-1, to=5-6]
	\arrow[from=5-6, to=1-6]
	\arrow[curve={height=-18pt}, from=5-6, to=3-3]
\end{tikzcd}\]


We noted above that if $\alpha < \lambda$ then since $f$ and $\mathfrak{N}_\alpha$ are both elements of $\mathfrak{N}_{\alpha+1}$, it makes sense to form the diagram $\mathcal{D}^{r^{f,\mathfrak{N}_\alpha}, \mathfrak{N}_{\alpha+1}}$.  This diagram is displayed below, with 
\[
\tau^{\alpha+1}:=r^{r^{f,\mathfrak{N}_\alpha}, \mathfrak{N}_{\alpha+1}} \text{ and } 
\varphi^{\alpha,\alpha+1}:= \widetilde{r^{f,\mathfrak{N}_\alpha} \restriction \mathfrak{N}_{\alpha+1}},
\]
and all arrows are inclusions:

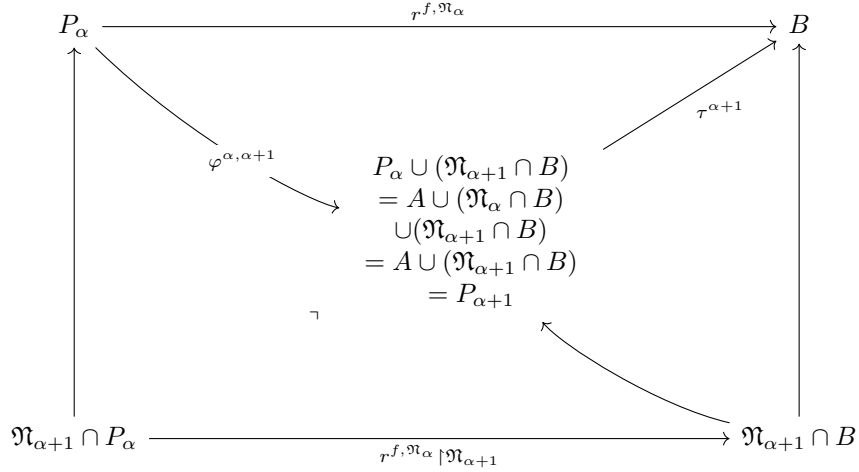
\[\begin{tikzcd}
	{P_\alpha} &&&&& B \\
	\\
	&&& \begin{array}{c} P_\alpha \cup (\mathfrak{N}_{\alpha+1} \cap B) \\ = A \cup (\mathfrak{N}_\alpha \cap B) \\ \cup (\mathfrak{N}_{\alpha+1} \cap B) \\=A \cup (\mathfrak{N}_{\alpha+1} \cap B) \\=P_{\alpha+1} \end{array} \\
	\\
	{\mathfrak{N}_{\alpha+1} \cap P_\alpha} &&&&& {\mathfrak{N}_{\alpha+1} \cap B}
	\arrow["{r^{f,\mathfrak{N}_\alpha}}", from=1-1, to=1-6]
	\arrow["{\varphi^{\alpha,\alpha+1}}"{description}, curve={height=18pt}, from=1-1, to=3-4]
	\arrow["{\tau^{\alpha+1}}"', from=3-4, to=1-6]
	\arrow["\ulcorner"{anchor=center, pos=0.125, rotate=-90}, draw=none, from=3-4, to=5-1]
	\arrow[from=5-1, to=1-1]
	\arrow["{r^{f,\mathfrak{N}_\alpha} \restriction \mathfrak{N}_{\alpha+1}}"', from=5-1, to=5-6]
	\arrow[from=5-6, to=1-6]
	\arrow[curve={height=-18pt}, from=5-6, to=3-4]
\end{tikzcd}\]

Notice that the pushout in the middle of that diagram, i.e., the act $A \cup (\mathfrak{N}_{\alpha+1} \cap B)$, is identical to the pushout in the diagram $\mathcal{D}^{f,\mathfrak{N}_{\alpha+1}}$.  So $\tau^{\alpha+1}$ is the same inclusion as $r^{f,\mathfrak{N}_{\alpha+1}}$.



For $\alpha \le \gamma \le \text{cf}(\lambda)$ define $\varphi^{\alpha,\gamma}$ to be the inclusion $P_\alpha \subseteq P_\gamma$, and define $\tau^{\alpha}$ to be the inclusion $P_\alpha \subseteq B$.  Note this coheres with the previous definitions in the particular case $\gamma = \alpha+1$.  Continuity of the $\vec{\mathfrak{N}}$ sequence ensures that for limit ordinals $\gamma$, 
\[
P_\gamma = A \cup (\mathfrak{N}_\gamma \cap B) = A \cup \left( \left( \bigcup_{\zeta < \gamma } \mathfrak{N}_\zeta \right) \cap B \right) = A \cup \bigcup_{\zeta < \gamma} (\mathfrak{N}_\zeta \cap B)= \bigcup_{\zeta < \gamma} \left( A \cup (\mathfrak{N}_\zeta \cap B) \right) =\bigcup_{\zeta < \gamma} P_\zeta
\]
Hence the inclusion $\varphi_{0,\text{cf}(\lambda)}: P_0 \to P_{\text{cf}(\lambda)}$ is a transfinite composition of the inclusions $\varphi^{\alpha,\alpha+1}: P_\alpha \subseteq P_{\alpha+1}$.  Note also that 
\[
f \restriction \mathfrak{N}_{\text{cf}(\lambda)} = f \text{ and } r^{f,\mathfrak{N}_{\text{cf}(\lambda)}} = \text{id}_B
\]
because $B \subset \mathfrak{N}_{\text{cf}(\lambda)}$.  Hence, $f$ can be viewed as the composition 
\[
f = \varphi^{0,\text{cf}(\lambda)} \circ \widetilde{f \restriction \mathfrak{N}_0}
\]
as displayed in the following commutative diagram:

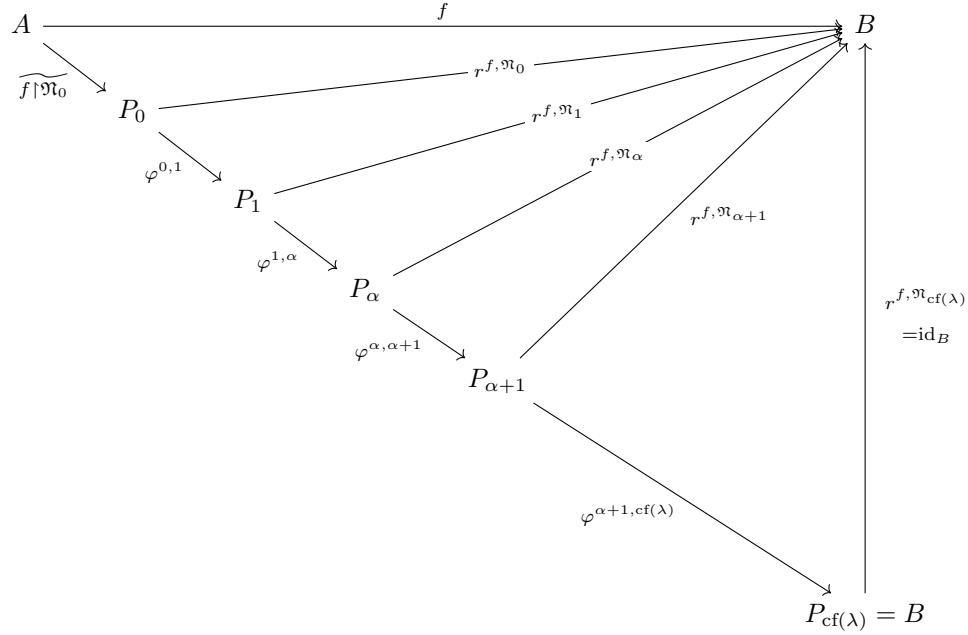
\[\begin{tikzcd}
	A &&&&&&&& B \\
	& {P_0} \\
	&& {P_1} \\
	&&& {P_\alpha} \\
	&&&& {P_{\alpha+1}} \\
	\\
	\\
	\\
	&&&&&&&& {P_{\text{cf}(\lambda)}=B}
	\arrow["f", from=1-1, to=1-9]
	\arrow["{\widetilde{f \restriction \mathfrak{N}_0}}"', from=1-1, to=2-2]
	\arrow["{r^{f,\mathfrak{N}_0}}"{description}, from=2-2, to=1-9]
	\arrow["{\varphi^{0,1}}"', from=2-2, to=3-3]
	\arrow["{r^{f,\mathfrak{N}_1}}"{description}, from=3-3, to=1-9]
	\arrow["{\varphi^{1,\alpha}}"', from=3-3, to=4-4]
	\arrow["{r^{f,\mathfrak{N}_\alpha}}"{description}, from=4-4, to=1-9]
	\arrow["{\varphi^{\alpha,\alpha+1}}"', from=4-4, to=5-5]
	\arrow["{r^{f,\mathfrak{N}_{\alpha+1}}}"', from=5-5, to=1-9]
	\arrow["{\varphi^{\alpha+1,\text{cf}(\lambda)}}"', from=5-5, to=9-9]
	\arrow["\begin{array}{c} r^{f,\mathfrak{N}_{\text{cf}(\lambda)}} \\=\text{id}_B \end{array}"', from=9-9, to=1-9]
\end{tikzcd}\]

Recall that for each $\alpha < \text{cf}(\lambda)$, the inclusion $\varphi^{\alpha,\alpha+1}: P_\alpha \subset P_{\alpha+1}$ is a pushout of $r^{f,\mathfrak{N}_\alpha} \restriction \mathfrak{N}_{\alpha+1}$.  And by \eqref{eq_quot_in_M}, $r^{f,\mathfrak{N}_\alpha} \restriction \mathfrak{N}_{\alpha+1}$ is a member of $\mathcal{M}$; moreover, its codomain is $\mathfrak{N}_{\alpha+1} \cap B$, which is of size $\le |\mathfrak{N}_{\alpha+1}| < \lambda$.  So the induction hypothesis ensures that
\begin{equation}
\forall \alpha < \text{cf}(\lambda) \ \ \varphi^{\alpha,\alpha+1} \in \text{cell}(\mathcal{M}^{<\kappa}).
\end{equation}

Also, the very first step $\widetilde{f \restriction \mathfrak{N}_0}$ of the factorization of $f$ is a pushout of $f \restriction \mathfrak{N}_0$, which is a member of $\mathcal{M}$ by \eqref{eq_rest_in_M} and has codomain of size $\le |\mathfrak{N}_0| < \lambda$.  So $\widetilde{f \restriction \mathfrak{N}_0}$ is also in $\text{cell}(\mathcal{M}^{<\kappa})$ by the induction hypothesis.

This concludes the proof that $f \in \text{cell}(\mathcal{M}^{<\kappa})$.

\ref{item_CellGen} $\implies$ \ref{item_AE}:  Suppose $\mathcal{M}$ is cellularly generated.  Then $\mathcal{M} = \text{cell}(\mathcal{M}_0)$ for some sub\textbf{set} $\mathcal{M}_0$ of $\mathcal{M}$.  Since $\mathcal{M}$ consists of monomorphisms by assumption, we can without loss of generality (enlarging $\mathcal{M}_0$ if necessary) assume $\mathcal{M} = \text{cell}\left( \mathcal{M}^{<\kappa} \right)$ for some regular $\kappa  > |S| + \aleph_0$, where $\mathcal{M}^{<\kappa}$ is a representative set of those members of $\mathcal{M}$ whose codomain has cardinality $<\kappa$.  We claim that the parameter
\[
p:= \big( \mathcal{M}^{<\kappa}, \kappa \big)
\]
witnesses that $\mathcal{M}$ is $\kappa$-a.e.\ effective.  To prove it, assume
\[
S \cup \{ S, (\mathcal{M}^{<\kappa}, \kappa) \} \subset \mathfrak{N} \prec_* (V,\in)
\]
and $\mathfrak{N} \cap \kappa$ is transitive.  Suppose $f: A \to B$ is a member of $ \text{cell}(\mathcal{M}^{<\kappa})$ with $f \in \mathfrak{N}$; we need to prove that $f \restriction \mathfrak{N}$ and $r^{f,\mathfrak{N}}$ are both members of $\text{cell}(\mathcal{M}^{<\kappa})$.  Since $f \in \text{cell}(\mathcal{M}^{<\kappa})$ there is a transfinite composition
\[
\vec{f} = \langle f_{\alpha,\beta}:A_\alpha \to A_\beta \ : \ \alpha \le \beta \le \mu \rangle
\]
such that $f = f_{0,\mu}$ and for each $\alpha < \mu$, $f_{\alpha,\alpha+1}$ appears in a pushout square $\mathcal{S}_\alpha$ of the following form:
 
\begin{equation}\label{eq_PO_Salpha}
\mathcal{S}_\alpha: \ \ \ \ 
\begin{tikzcd}
	{D_\alpha} && {C_\alpha} \\
	\\
	{A_\alpha} && {A_{\alpha+1}}
	\arrow["{m_\alpha \in \mathcal{M}^{<\kappa}}", from=1-1, to=1-3]
	\arrow["{\varphi_\alpha}"', from=1-1, to=3-1]
	\arrow["{\psi_\alpha}", from=1-3, to=3-3]
	\arrow["{f_{\alpha,\alpha+1}}"', from=3-1, to=3-3]
	\arrow["\ulcorner"{anchor=center, pos=0.125, rotate=180}, draw=none, from=3-3, to=1-1]
\end{tikzcd}
\end{equation}

By elementarity of $\mathfrak{N}$ and the assumption that $f$ and $\mathcal{M}^{<\kappa}$ are both elements of $ \mathfrak{N}$, we can take the sequences $\vec{f}$ and $\vec{\mathcal{S}}$ to both be \emph{elements} of $\mathfrak{N}$.  Observe also that if $\alpha \in \mathfrak{N}$, then $\mathcal{S}_\alpha \in \mathfrak{N}$ too, and hence both $D_\alpha$ and $C_\alpha$ are $<\kappa$-sized elements of $\mathfrak{N}$.  Then by Fact \ref{fact_element_subset} they're both subsets of $\mathfrak{N}$.  Then, since $\varphi_\alpha$ and $\psi_\alpha$ are elements of $\mathfrak{N}$ and their domains are contained in $\mathfrak{N}$, their images are contained in $\mathfrak{N}$ too.  To summarize:
\begin{equation}\label{eq_summary_alpha_in_N}
\alpha \in \mathfrak{N} \cap \mu \ \implies \ C_\alpha, \ D_\alpha, \text{im} \varphi_\alpha, \text{ and } \text{im} \psi_\alpha \text{ are all subsets of } \mathfrak{N}.
\end{equation} 

To see that $f_{0,\mu} \restriction \mathfrak{N}$ is in $\text{cell}(\mathcal{M}^{<\kappa})$:  the continuity of $\vec{f}$ and elementarity of $\mathfrak{N}$ imply that $f_{0,\mu} \restriction \mathfrak{N}$ is the colimit of the continuous system
\[
\vec{f} \restriction \mathfrak{N}:= \langle f_{\alpha,\beta} \restriction (\mathfrak{N} \cap A_\alpha) \to (\mathfrak{N} \cap A_\beta) \ : \ \alpha \le \beta \text{ are both elements of } \mathfrak{N} \rangle
\]
indexed by $\mathfrak{N} \cap \mu$.  We need to prove that its cells consist of pushouts of members of $\mathcal{M}^{<\kappa}$.  Suppose $\alpha \in \mathfrak{N} \cap \mu$.  Then by \eqref{eq_summary_alpha_in_N}, the square $\mathcal{S}_\alpha$ factors as:

\[\begin{tikzcd}
	\begin{array}{c} D_\alpha \\(=\mathfrak{N} \cap D_\alpha) \end{array} && \begin{array}{c} C_\alpha \\ (=\mathfrak{N} \cap C_\alpha) \end{array} \\
	\\
	{\mathfrak{N} \cap A_\alpha} && {\mathfrak{N} \cap A_{\alpha+1}} \\
	\\
	{A_\alpha} && {A_{\alpha+1}}
	\arrow["{m_\alpha \in \mathcal{M}^{<\kappa}}", from=1-1, to=1-3]
	\arrow[from=1-1, to=3-1]
	\arrow["{\varphi_\alpha}"', curve={height=30pt}, from=1-1, to=5-1]
	\arrow[from=1-3, to=3-3]
	\arrow["{\psi_\alpha}", curve={height=-30pt}, from=1-3, to=5-3]
	\arrow["{f_{\alpha,\alpha+1} \restriction \mathfrak{N}}"', from=3-1, to=3-3]
	\arrow["\subseteq", from=3-1, to=5-1]
	\arrow["\subseteq"', from=3-3, to=5-3]
	\arrow["{f_{\alpha,\alpha+1}}"', from=5-1, to=5-3]
\end{tikzcd}\]

The top square is just the restriction of the pushout square $\mathcal{S}_\alpha$ to $\mathfrak{N}$, and hence is also a pushout square, by Lemma \ref{lem_restrictColimit}.  This completes the proof that $f \restriction \mathfrak{N} \in \text{cell}\left( \mathcal{M}^{<\kappa} \right)$.

Now we prove that $r^{f_{0,\mu},\mathfrak{N}}$ is in $\text{cell} \left( \mathcal{M}^{<\kappa} \right)$.  Since $\mathcal{M}$ consists of monomorphisms, we can without loss of generality assume each $f_{\alpha,\alpha+1}$ is an inclusion.  Then we can view $r^{f_{0,\mu}, \mathfrak{N}}$ as the inclusion from $A_0 \cup (\mathfrak{N} \cap A_\mu)$ into $A_\mu$, which factors via intermediate acts $Q_\alpha$ as displayed below:
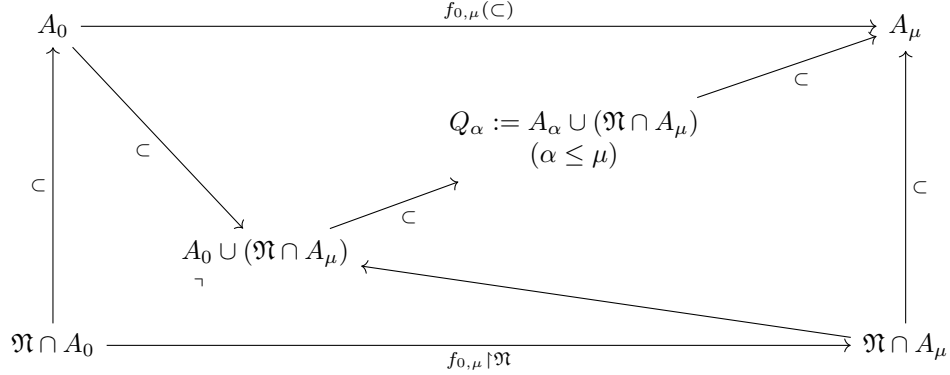
\[\begin{tikzcd}
	{A_0} &&&& {A_\mu} \\
	&& \begin{array}{c} Q_\alpha:=A_\alpha \cup (\mathfrak{N} \cap A_\mu) \\ (\alpha \le \mu) \end{array} \\
	& {A_0 \cup(\mathfrak{N} \cap A_\mu)} \\
	{\mathfrak{N} \cap A_0} &&&& {\mathfrak{N} \cap A_\mu}
	\arrow["{f_{0,\mu}  (\subset)}", from=1-1, to=1-5]
	\arrow["\subset"', from=1-1, to=3-2]
	\arrow["\subset"', from=2-3, to=1-5]
	\arrow["\subset"', from=3-2, to=2-3]
	\arrow["\ulcorner"{anchor=center, pos=0.125, rotate=-90}, draw=none, from=3-2, to=4-1]
	\arrow["\subset", from=4-1, to=1-1]
	\arrow["{f_{0,\mu} \restriction \mathfrak{N}}"', from=4-1, to=4-5]
	\arrow["\subset"', from=4-5, to=1-5]
	\arrow[from=4-5, to=3-2]
\end{tikzcd}\]

We need to prove that each inclusion $Q_\alpha \to Q_{\alpha+1}$ is a pushout of a member of $\mathcal{M}^{<\kappa}$.  There are two cases.

\textbf{Case 1: $\alpha \in \mathfrak{N}$.}  Then by \eqref{eq_summary_alpha_in_N}, $\text{im} \psi_\alpha \subseteq \mathfrak{N}$.  Since $A_{\alpha+1}$ is the pushout of the diagram $\mathcal{S}_\alpha$ from \eqref{eq_PO_Salpha}, and since a pushout in a category of $S$-acts is contained in the union of the maps pointing to it (\cite[p.54]{MR3249868}),
 \[
A_{\alpha+1} = A_\alpha \cup \text{im} \psi_\alpha \subseteq A_\alpha \cup (\mathfrak{N} \cap A_\mu). 
 \]
Hence, $Q_\alpha = Q_{\alpha+1}$ in this case, so the inclusion is simply the identity.  If $\mathcal{M}$ contains all isomorphisms we are done; otherwise this cell is redundant and can simply be deleted from the transfinite composition.

\textbf{Case 2:  $\alpha \notin \mathfrak{N}$.}  Then by elementarity of $\mathfrak{N}$, $\alpha+1 \notin \mathfrak{N}$ too (since $\alpha$, the ordinal predecessor of $\alpha+1$, is definable from the parameter $\alpha+1$ in $(V,\in)$).  Suppose $x \in \mathfrak{N} \cap A_{\alpha+1}$.  Then $x \in \mathfrak{N} \cap A_\mu$, so the least ordinal $\zeta_x$ such that $x \in A_{\zeta_x}$ is an element of $\mathfrak{N}$, by elementarity.  Then $\zeta_x \le \alpha+1$ because $x \in A_{\alpha+1}$, but in fact $\zeta_x < \alpha$ because neither $\alpha$ nor $\alpha+1$ are elements of $\mathfrak{N}$, while $\zeta_x \in \mathfrak{N}$.  In particular, $x \in A_\alpha$.  To summarize: our case assumption implies
 \begin{equation}\label{eq_SameIntersect}
 \mathfrak{N} \cap A_\alpha = \mathfrak{N} \cap A_{\alpha+1}.
 \end{equation}
 
Consider the commutative diagram below, whose first square is the pushout $\mathcal{S}_\alpha$:
\[\begin{tikzcd}
	{D_\alpha} && {A_{\alpha+1}} && \begin{array}{c} Q_{\alpha+1} \\= A_{\alpha+1} \cup (\mathfrak{N} \cap A_\mu) \end{array} \\
	\\
	{C_\alpha} && {A_\alpha} && \begin{array}{c} Q_\alpha\\=A_\alpha \cup (\mathfrak{N} \cap A_\mu) \end{array}
	\arrow["{\psi_\alpha}", from=1-1, to=1-3]
	\arrow["\subset"', from=1-3, to=1-5]
	\arrow["\ulcorner"{anchor=center, pos=0.125, rotate=-90}, draw=none, from=1-3, to=3-1]
	\arrow["{m_\alpha \in \mathcal{M}^{<\kappa}}", from=3-1, to=1-1]
	\arrow["{\varphi_\alpha}"', from=3-1, to=3-3]
	\arrow["\subset"', from=3-3, to=1-3]
	\arrow["\subset"', from=3-3, to=3-5]
	\arrow["\subset"', from=3-5, to=1-5]
\end{tikzcd}\]
We claim the outer rectangle is a pushout.  This will complete the proof, since it will imply the inclusion $Q_\alpha \to Q_{\alpha+1}$ is a pushout of $m_\alpha$.  

By the pasting lemma it suffices to prove that the right square is a pushout.   Assume $k:A_{\alpha+1} \to Y$ and $h: Q_\alpha \to Y$ are morphisms that commute with the inclusions $A_\alpha \subset Q_\alpha$ and $A_\alpha \subset A_{\alpha+1}$.  We claim there is a unique $\tau$ making the following diagram commute:

\[\begin{tikzcd}
	&&& Y \\
	{A_{\alpha+1}} && \begin{array}{c} Q_{\alpha+1} \\= A_{\alpha+1} \cup (\mathfrak{N} \cap A_\mu) \end{array} \\
	\\
	{A_\alpha} && \begin{array}{c} Q_\alpha\\=A_\alpha \cup (\mathfrak{N} \cap A_\mu) \end{array}
	\arrow["k", from=2-1, to=1-4]
	\arrow["\subset"', from=2-1, to=2-3]
	\arrow["\tau"{description}, dashed, from=2-3, to=1-4]
	\arrow["\subset"', from=4-1, to=2-1]
	\arrow["\subset"', from=4-1, to=4-3]
	\arrow["h"', from=4-3, to=1-4]
	\arrow["\subset"', from=4-3, to=2-3]
\end{tikzcd}\]

Clearly there is at most one  possibility for $\tau$, namely, $\tau(x) = k(x)$ for $x \in A_{\alpha+1}$, and $\tau(x) = h(x)$ for $y \in \mathfrak{N} \cap A_\mu$.  We need to show this is a well-defined function (if it is, then easily it is a morphism).  Suppose $x \in A_{\alpha+1} \cap (\mathfrak{N} \cap A_\mu)$.  Then by \eqref{eq_SameIntersect}, $x \in \mathfrak{N} \cap A_\alpha$; hence, $h(x) = k(x)$ by commutativity of the outer rectangle.

\end{proof}

\begin{corollary}\label{cor_IntersectSetMany}
Suppose $\langle \mathcal{M}_i \ : \ i \in I \rangle$ is a set-indexed collection of cellularly-generated classes of monomorphisms in $\overline{\text{Act}}$-$S$ (or Act-$S$, or Act$_0$-$S$) each containing all isomorphisms.  Then $\bigcap_{i \in I} \mathcal{M}_i$ is cellularly-generated.\footnote{If $I$ is infinite, to avoid metamathematical issues one needs to assume the intersection is first-order definable in $(V,\in)$.  This will hold, e.g., if there is a uniform definition of the $\mathcal{M}_i$'s. }
\end{corollary}
\begin{proof}
By the \ref{item_CellGen} $\implies$ \ref{item_AE} direction of Theorem \ref{thm_NewChar}, each $\mathcal{M}_i$ is $\kappa_i$-a.e.\ effective for some $\kappa_i$.   Since $I$ is a set, there is a regular $\kappa$ strictly larger than all the $\kappa_i$'s and also larger than $ |I| + \aleph_0$.  So for each $i \in I$, $\mathcal{M}_i$ is $\kappa$-a.e.\ effective.  So, using the notation introduced after Definition \ref{def_EffectiveSubobj}:

\begin{equation}\label{eq_kappa_i_eff_all_i}
\forall i \in I \ \ (\text{aa}^*_\kappa \ \mathfrak{N}) \ \left( \forall f \in \mathfrak{N} \cap \mathcal{M}_i \right) \left( f \restriction \mathfrak{N} \in \mathcal{M}_i \text{ and } r^{f,\mathfrak{N}} \in \mathcal{M}_i \right).
\end{equation}
Recall the $\text{aa}_\kappa^* \mathfrak{N}$ quantifier here means that there is a parameter $p_i$ such that the rest of the statement holds whenever $S \cup \{ S,p_i \} \subset \mathfrak{N} \prec_* (V,\in)$ and $\mathfrak{N} \cap \kappa$ is transitive. 

Let $\vec{p} = \langle p_i \ : \ i \in I \rangle$; we claim that $\vec{p}$ and $\kappa$ witness that $\bigcap_{i \in I} \mathcal{M}_i$ is $\kappa$-a.e.\ effective.  Suppose $S \cup \{ S, \vec{p} \} \subset \mathfrak{N} \prec_* (V,\in)$ and $\mathfrak{N} \cap \kappa$ is transitive, and consider any $f \in \mathfrak{N} \cap \bigcap_{i \in I} \mathcal{M}_i$.  Since $I= \text{dom}(\vec{p})$ is a $<\kappa$-sized element of $\mathfrak{N}$ and $\mathfrak{N} \cap \kappa$ is transitive, $I$ is also a subset of $\mathfrak{N}$ by Fact \ref{fact_element_subset}.  Then $\vec{p}(i)=p_i$ is an element of $\mathfrak{N}$ by elementarity, for each $i \in I$.  So by \eqref{eq_kappa_i_eff_all_i} we conclude that for each $i$, $f \restriction \mathfrak{N}$ and $r^{f,\mathfrak{N}}$ are both in $\mathcal{M}_i$.  Therefore $f \restriction \mathfrak{N}$ and $r^{f,\mathfrak{N}}$ are both in $\bigcap_{i \in I} \mathcal{M}_i$.
\end{proof}

\section{Proof of Theorems \ref{thm_Fmono_cofibgen} and \ref{thm_LO}}

A monomorphism $f:X \to Y$ of (right) $S$-acts is \textbf{pure} if every finite system of equations with parameters from $f(X)$ that is solvable in $Y$, is also solvable in $f(X)$.  $f$ is \textbf{stable} if for every homomorphism $\lambda: A \to B$ of left $S$-acts, 
\[
\text{im} (f \otimes \text{id}_B) \cap \text{im} (\text{id}_Y \otimes \lambda) \subseteq \text{im}(f \otimes \lambda).
\]
Purity implies stability.

\begin{lemma}\label{lem_ElementaryIsPure}
Suppose $S$ is a monoid and 
\[
S \cup \{S \} \subset \mathfrak{N} \prec_* (V,\in).
\]
Then for any $S$-act $A$ with $A \in \mathfrak{N}$, the inclusion $\mathfrak{N} \cap A \to A$ is pure (and hence stable). 
\end{lemma}
\begin{proof}
The fact that $S$ is both an element and subset of $\mathfrak{N}$, together with the assumption that $\mathfrak{N} \prec_* (V,\in)$, ensure that $\mathfrak{N} \cap A$ is an elementary subact of $A$ in the language of $S$-acts (which ensures purity).  The proof is similar to the module case in Lemma 2.4 of \cite{Cox_ApproxSurvey}.  
\end{proof}

\begin{lemma}\label{lem_Renshaw}\cite[Lemma 2.1]{MR1899443}
Let $f: X \to Y$ be a monomorphism of right nonempty $S$-acts.
\begin{enumerate}
 \item\label{item_FlatQuotImplies} If $Y/X$ is flat then $S$ is right-reversible and $f$ is stable.\footnote{Lemma 2.1 of  \cite{MR1899443} refers to  \cite[Corollary 4.9]{MR0829389} for the proof, but the statement of the latter includes the background assumption that $Y$ is flat, even for part \eqref{item_FlatQuotImplies}.  This background assumption is unnecessary for this particular claim, since the proof of the ``only if" direction of Theorem 4.8 of \cite{MR0829389} does not use the assumption that $1 \otimes \lambda$ is one-to-one (that assumption is only used for the ``if" direction of Theorem 4.8).  We thank James Renshaw for clarification on this issue.}

 \item If $Y$ is flat, $S$ is right-reversible, and $f$ is stable, then $Y/X$ is flat.\footnote{We caution that the statement of the corresponding lemma in \cite{MR3249868} has a misprint; the third sentence of Lemma 1.3 of \cite{MR3249868} should have, as an assumption, that the monoid is right-reversible.}
\end{enumerate}
\end{lemma}

\begin{corollary}\label{cor_TraceFlat}
Suppose $S$ is a monoid, $F$ is a flat $S$-act,  \[
S \cup \{S \} \subset \mathfrak{N} \prec_* (V,\in),
\]
and $F \in \mathfrak{N}$.  Then $\mathfrak{N} \cap F$ is flat.  The quotient $\frac{F}{\mathfrak{N} \cap F}$ is flat if, and only if, $S$ is right-reversible.
\end{corollary}
\begin{proof}
By Lemma \ref{lem_ElementaryIsPure}, $\mathfrak{N} \cap F$ is pure in $F$, and pure subacts of flat acts are flat.  Purity of the inclusion $\mathfrak{N} \cap F \to F$ implies that this inclusion is stable.  Then $F/(\mathfrak{N} \cap F)$ is flat if and only if $S$ is right-reversible by Lemma \ref{lem_Renshaw}. 
\end{proof}

\subsection{Proof of Theorem \ref{thm_Fmono_cofibgen}}
For the \eqref{item_S_rr} $\implies$ \eqref{item_F_Mono_cofibGen} direction, assume $S$ is right-reversible and $\mathcal{F}$-Mono is closed under compositions.  Then $\mathcal{F}$-Mono is cofibrantly closed (see \cite[Theorem 3.11]{MR3249868}; the proof of part 2 there works if $\mathcal{F}$-Mono is closed under compositions).  Since it is cofibrantly closed, to prove that it is cofibrantly generated it suffices to prove it is cellularly generated.  Let $\kappa$ be any regular cardinal larger than $|S|+\aleph_0$; we claim that $\mathcal{F}$-mono is $\kappa$-a.e.\ effective, which by the \ref{item_AE} $\implies$ \ref{item_CellGen} direction of Theorem \ref{thm_NewChar}, will complete the proof of this direction of Theorem \ref{thm_Fmono_cofibgen}.

We consider a single, arbitrary $\mathfrak{N}$ such that $S \cup \{S \} \subset \mathfrak{N} \prec_* (V,\in)$.  We do not actually need to assume that $\mathfrak{N} \cap \kappa$ is transitive for the current application, or to include any other parameters.   We need to prove that if $f \in \mathcal{F}$-Mono with $f \in \mathfrak{N}$, then both $f \restriction \mathfrak{N}$ and $r^{f,\mathfrak{N}}$ are in $\mathcal{F}$-Mono.  Assume without loss of generality that $f$ is an inclusion $A \subset B$, with $B/A$ flat.  Since $A \cap (\mathfrak{N} \cap B) = \mathfrak{N} \cap A$, the pushout of the span $A \leftarrow \mathfrak{N} \cap A \rightarrow \mathfrak{N} \cap B$ is exactly $A \cup (\mathfrak{N} \cap B)$; so the diagram $\mathcal{D}^{f,\mathfrak{N}}$ has the following form, with all maps inclusions:

\[\begin{tikzcd}
	A &&&&& B \\
	\\
	&& {A \cup (\mathfrak{N} \cap B)} \\
	\\
	{\mathfrak{N} \cap A} &&&&& {\mathfrak{N} \cap B}
	\arrow["f"{description}, from=1-1, to=1-6]
	\arrow[curve={height=18pt}, from=1-1, to=3-3]
	\arrow["{r^{f,\mathfrak{N}}}"{description}, from=3-3, to=1-6]
	\arrow[from=5-1, to=1-1]
	\arrow["{f \restriction \mathfrak{N}}"{description}, from=5-1, to=5-6]
	\arrow[from=5-6, to=1-6]
	\arrow[curve={height=-18pt}, from=5-6, to=3-3]
\end{tikzcd}\]

Clearly $f \restriction\mathfrak{N}$ and $r^{f,\mathfrak{N}}$ are monomorphisms, so we just need to show their Rees quotients are flat.  Since $f \in \mathfrak{N}$, its domain $A$ and codomain $B$ are elements of $\mathfrak{N}$, so by elementarity of $\mathfrak{N}$, the Rees quotient $F:=B/A$ is also an element of $\mathfrak{N}$.  By Lemma \ref{lem_HomThm},
\begin{equation}\label{eq_TraceOnQuot}
\frac{\mathfrak{N} \cap B}{\mathfrak{N} \cap A} \simeq \mathfrak{N} \cap \frac{B}{A}  = \mathfrak{N} \cap F
\end{equation}
and
\begin{equation}\label{eq_QuotOfTrace}
\frac{B}{A \cup (\mathfrak{N}\cap B)} \simeq \frac{B/A}{\mathfrak{N} \cap (B/A)} = \frac{F}{\mathfrak{N} \cap F}
\end{equation}
By Corollary \ref{cor_TraceFlat}, the right-reversibility of $S$, the assumption that $F$ is flat, and the fact that $F$ is an element of $\mathfrak{N}$, it follows that the right sides of \eqref{eq_TraceOnQuot} and \eqref{eq_QuotOfTrace} are both flat.   Hence, both $f \restriction \mathfrak{N}$ and $r^{f,\mathfrak{N}}$ are in $\mathcal{F}$-mono.

To prove the \eqref{item_F_Mono_cofibGen} $\implies$ \eqref{item_S_rr} direction of Theorem \ref{thm_Fmono_cofibgen}, assume $\mathcal{F}$-Mono is cofibrantly generated in $\overline{\text{Act}}$-$S$.  In particular, it is closed under composition, and by Fact \ref{fact_CharCofClos}, 
\[
\mathcal{F}\text{-Mono} = \text{Rt } \underbrace{\text{cell}(\mathcal{M}_0)}_{\mathcal{C}:=}
\]
for some set $\mathcal{M}_0 \subset \mathcal{F}$-Mono.  By the \ref{item_CellGen} $\implies$ \ref{item_AE} direction of Theorem \ref{thm_NewChar} there is a $\kappa$ such that $\mathcal{C}$ is $\kappa$-a.e.\ effective.  Let $p$ be the parameter witnessing this and fix any $\mathfrak{N}$ such that $\mathfrak{N} \cap \kappa$ is transitive and
\begin{equation}\label{eq_ParameterInN}
S \cup \{ S,p,f \} \subset \mathfrak{N} \prec_* (V,\in),
\end{equation}
where $f: A \to B$ is any fixed member of $\mathcal{C}$.  Now $\mathcal{C}  = \text{cell} (\mathcal{M}_0)$ and $\mathcal{M}_0$ is contained in the cellularly-closed class $\mathcal{F}$-Mono.  So, in particular, $f \in \mathcal{F}$-Mono, so $B/A$ is flat.  If $A \ne \emptyset$ then we conclude $S$ is right-reversible by Lemma \ref{lem_Renshaw}.  If $A = \emptyset$ then $B$ must be flat, and since $f: \emptyset \to B$ is a member of $\mathfrak{N} \cap \mathcal{C}$, the map $r^{f,\mathfrak{N}}$ from $P^{f,\mathfrak{N}} = \emptyset \cup (\mathfrak{N} \cap B)= \mathfrak{N} \cap B$ into $B$ is a member of $\mathcal{C}$; in particular it is in $\mathcal{F}$-Mono and hence $\frac{B}{\mathfrak{N} \cap B}$ is flat.  Now $\mathfrak{N} \cap B$ is nonempty (by elmentarity of $\mathfrak{N}$) and hence again we conclude by Lemma \ref{lem_Renshaw} that $S$ is right-reversible.

The proof of the version of Theorem \ref{thm_Fmono_cofibgen} for Act$_0$-$S$ (when $S$ has a left zero) is similar.

\subsection{Proof of Theorem \ref{thm_LO}   } 

Assume $S$ is a right-LO monoid, and $\mathcal{F}$-PureMono is the class of pure monomorphisms in $\overline{\text{Act}}$-$S$ with flat Rees quotient.  So
\[
\mathcal{F}\text{-PureMono} = \left( \mathcal{F}\text{-Mono} \right) \cap \left( \text{PureMono} \right).
\]

Now PureMono is always cofibrantly closed (\cite{MR2280436}).  And although we don't know if $\mathcal{F}$-Mono is even closed under compositions, $\mathcal{F}$-PureMono is always closed under compositions (\cite[Theorem IV.1.5]{renshaw1986flatness}).  It then follows easily that $\mathcal{F}$-PureMono is cofibrantly closed.  So to prove it is cofibrantly generated, it suffices to prove that it is cellularly generated; and by Theorem \ref{thm_NewChar} it suffices to show that $\mathcal{F}$-PureMono is a.e.\ effective.

The assumption that $S$ is right-LO implies PureMono is cofibrantly generated in $\overline{\text{Act}}$-$S$; see \cite[Corollary 2.6]{MR2280436} and \cite{MR4121092} (or \cite{Cox_et_al_CofGenPureMon}).  By Theorem \ref{thm_NewChar} there is a $\kappa$ such that PureMono is $\kappa$-a.e.\ effective (in fact, it satisfies the stronger notion of Borceux-Rosick\'y).  So there is a parameter $p$ such that whenever
\[
S \cup \{ S,p \} \subset \mathfrak{N} \prec_* (V,\in)
\]
and $\mathfrak{N} \cap \kappa$ is transitive, then for any $f: A \to B$ that is in $\mathfrak{N} \cap \text{PureMono}$, both $f \restriction \mathfrak{N}$ and $r^{f,\mathfrak{N}}$ are pure.  

If, in addition, $f$ has flat Rees quotient, then by Lemma \ref{lem_HomThm} (assuming WLOG $f$ is an inclusion) we have the isomorphisms
\[
\frac{\mathfrak{N} \cap B}{\mathfrak{N} \cap A} \simeq \mathfrak{N} \cap \frac{B}{A} 
\]
and
\[
\frac{B}{A \cup (\mathfrak{N}\cap B)} \simeq \frac{B/A}{\mathfrak{N} \cap (B/A)},
\]
The assumption that $S$ is right-LO implies, in particular, that $S$ is right-reversible.  So again, as in that proof of Theorem \ref{thm_Fmono_cofibgen}, since $B/A$ is a flat element of $\mathfrak{N}$, it follows from Corollary \ref{cor_TraceFlat} that the right sides of both of those isomorphisms are flat.  So the Rees quotients of $f \restriction \mathfrak{N}$ and $r^{f,\mathfrak{N}}$ are both flat.

This completes the proof that $\mathcal{F}$-PureMono is cofibrantly generated in $\overline{\text{Act}}$-$S$ (assuming $S$ is right-LO).  This implies FCC in Act-$S$ by Lemma \ref{lem_CellGenExpandedGivesCovers}.  The proof for Act$_0$-$S$ is similar.

\section{Proof of Theorem \ref{thm_SF}}

As mentioned in the introduction, cofibrant generation of $\mathcal{SF}$-Mono appears to be much stronger than cofibrant generation of $\mathcal{F}$-Mono, because the former implies a bound on the cardinality of the indecomposable members of $\mathcal{SF}$.  This appears to be due mainly to two facts (Lemmas \ref{lem_SF_locallyCyc} and \ref{lem_SFquot_PureInc}) that distinguish $\mathcal{SF}$ from $\mathcal{F}$.  An $S$-act $A$ is \textbf{locally cyclic} if for every $a,b \in A$ there is a $c \in A$ and $s,t \in S$ such that $cs=a$ and $ct=b$.  Clearly every locally cyclic act is indecomposable, but for strongly flat acts we have the converse:  
\begin{lemma}\label{lem_SF_locallyCyc}\cite[Theorem 3.7]{MR1839214}
For any monoid $S$, the locally cyclic strongly flat acts are the same as the indecomposable strongly flat acts.
\end{lemma}

\begin{lemma}\label{lem_SFquot_PureInc}
(See Corollary 5.1(3) and the paragraph at the bottom of page 599 of \cite{MR3251751}).  Suppose $S$ is a monoid and $A \subset B$ is an inclusion of nonempty $S$-acts.  If $B/A$ is strongly flat, then the inclusion is pure.
\end{lemma}

\subsection{Proof of Theorem \ref{thm_SF},  \ref{item_SF_MonoCofibGen} $\implies$ \ref{item_SF_LeftCollapseIndec}}

Assume $\mathcal{SF}$-Mono is cofibrantly generated in $\overline{\text{Act}}$-$S$.  Then in particular it is closed under compositions, and by Fact \ref{fact_CharCofClos} there is a set $\mathcal{M}_0$ such that
\[
\mathcal{SF}\text{-Mono} =  \text{Rt} \ \underbrace{\text{cell}(\mathcal{M}_0)}_{\mathcal{C}:=}. 
\]
So by Theorem \ref{thm_NewChar} there is a cardinal $\kappa > |S|$ such that
\begin{equation}\label{eq_aa_W_SF_lc}
\text{aa}^*_\kappa \mathfrak{W} \left( \forall f \in \mathfrak{W} \cap \mathcal{C} \right) \left( f \restriction \mathfrak{W} \text{ and } r^{f,\mathfrak{W}} \text{ are in } \mathcal{C} \right)
\end{equation}
as witnessed by some parameter $p$.\footnote{Recall this means that if 
\[
S \cup \{ S,p \} \subset \mathfrak{W} \prec_* (V,\in)
\]
and $\mathfrak{W} \cap \kappa$ is transitive, then for every $f \in \mathfrak{W} \cap \mathcal{C}$, both $f \restriction \mathfrak{W}$ and $r^{f,\mathfrak{W}}$ are in $\mathcal{C}$. }

First we see why $S$ must be left-collapsible.  Fix any $f:A \to B$ in $\mathcal{C}$, and without loss of generality assume $f$ is an inclusion.  If $A \ne \emptyset$ then strong flatness of $B/A$ implies $S$ is left collapsible by \cite[Theorem 6.2(7)]{MR1899443}.  So assume $A = \emptyset$.  By Fact \ref{fact_LS_class} there is an $\mathfrak{N}$ such that $S \cup \{f, p,S \} \subset \mathfrak{N} \prec_* (V,\in)$ and $\mathfrak{N} \cap \kappa$ is transitive.    Since $A = \emptyset$ and the inclusion $A \subset B$ is in $\mathcal{C} \subset \mathcal{SF}$-Mono, then $B$ is strongly flat, and \eqref{eq_aa_W_SF_lc} implies $r^{f,\mathfrak{N}}$ is in $\mathcal{SF}$-Mono.  Since $A = \emptyset$, $r^{f,\mathfrak{N}}$ is just the inclusion from $\mathfrak{N} \cap B$ into $B$.  Then $\frac{B}{\mathfrak{N} \cap B}$ is strongly flat, and $\mathfrak{N} \cap B$ is nonempty by elementarity of $\mathfrak{N}$.  Then again by \cite[Theorem 6.2(7)]{MR1899443} this implies $S$ is left-collapsible.

It remains to show that all indecomposable members of $\mathcal{SF}$ have size at most $\kappa$.  By Lemma \ref{lem_SF_locallyCyc}, to prove the latter claim is equivalent to proving that the locally cyclic members of $\mathcal{SF}$ have size at most $\kappa$.  Suppose, toward a contradiction, that there is a locally cyclic, strongly flat act $F$ with $|F|>\kappa$.  Using Fact \ref{fact_LS_class} (with the cardinal $\kappa^+$) there is an $\mathfrak{M}$ such that $S \cup \{ p, S,F, \kappa \} \subset \mathfrak{M}  \prec_* (V,\in)$ and $|\mathfrak{M}|=\kappa \subset \mathfrak{M}$.  Using Fact \ref{fact_LS_class} again (this time with cardinal $\kappa$), there is an $\mathfrak{N}$ such that $S \cup \{p, S,F,\kappa, \mathfrak{M} \} \subset \mathfrak{N} \prec_* (V,\in)$, $|\mathfrak{N}|<\kappa$, and $\mathfrak{N} \cap \kappa$ is transitive.  We emphasize here that $\mathfrak{M}$ is an element of, but not a subset of, $\mathfrak{N}$.

Consider the empty monomorphism $f: \emptyset \subset F$, whose Rees quotient is just $F$ and hence is a member of $\mathcal{SF}$-Mono (and also a member of $\mathfrak{M}$, since $F \in \mathfrak{M}$).  By \eqref{eq_aa_W_SF_lc}, both $f \restriction \mathfrak{M}$ and $r^{f,\mathfrak{M}}$ are in $\mathcal{SF}$-Mono.  The pushout of the span $\emptyset \leftarrow \emptyset \rightarrow \mathfrak{M} \cap F$ is just $\mathfrak{M} \cap F$ and $r^{f,\mathfrak{M}}$ is just the inclusion $i: \mathfrak{M} \cap F \subset F$.  In summary, 
\begin{equation}\label{eq_i_M_in_SF_Mono}
\text{ The inclusion } i: \mathfrak{M} \cap F \to F \text{ is in } \mathcal{SF} \text{-mono}.
\end{equation}

Since $\mathfrak{M}$ and $F$ are elements of $\mathfrak{N}$, the inclusion $i: \mathfrak{M} \cap F \to F$ is an element of $\mathfrak{N}$, so it makes sense to form the diagram $\mathcal{D}^{i,\mathfrak{N}}$ from Definition \ref{def_MainDiag} displayed below:

\[\begin{tikzcd}
	{\mathfrak{M} \cap F} &&&&& F \\
	\\
	&& P \\
	\\
	{\mathfrak{N} \cap (\mathfrak{M} \cap F)} &&&&& {\mathfrak{N} \cap F}
	\arrow["{i \in \mathcal{SF} \text{-mono}}", from=1-1, to=1-6]
	\arrow[from=1-1, to=3-3]
	\arrow["{r^{i,\mathfrak{N}} }"', from=3-3, to=1-6]
	\arrow["\ulcorner"{anchor=center, pos=0.125, rotate=-90}, draw=none, from=3-3, to=5-1]
	\arrow["\subset", from=5-1, to=1-1]
	\arrow["{i \restriction \mathfrak{N} }"', from=5-1, to=5-6]
	\arrow["\subset"', from=5-6, to=1-6]
	\arrow[from=5-6, to=3-3]
\end{tikzcd}\]

Again by \eqref{eq_aa_W_SF_lc} (this time using the map $i$ in place of $f$), we conclude that $i \restriction \mathfrak{N}$ and $r^{i,\mathfrak{N}}$ are in $\mathcal{SF}$-mono.  Now $r^{i,\mathfrak{N}}$ is just the inclusion from $P=(\mathfrak{M} \cap F) \cup (\mathfrak{N} \cap F)$ into $F$; so
\begin{equation}
\frac{F}{P}=\frac{F}{(\mathfrak{M} \cap F) \cup (\mathfrak{N} \cap F)} \ \in \mathcal{SF}.
\end{equation}

By Lemma \ref{lem_SFquot_PureInc}, 

\begin{equation}\label{eq_P_to_F_pure}
\text{The inclusion } P \to F \text{ is pure.}
\end{equation}

We will use the local cyclicity of $F$, together with the properties of $\mathfrak{M}$ and $\mathfrak{N}$, to contradict \eqref{eq_P_to_F_pure}.  Notice:
\begin{itemize}
 \item  $|\mathfrak{M} \cap F| = \kappa$.  The $\le$ direction immediately follows from the fact that $|\mathfrak{M}|=\kappa$.  To see the $\ge$ direction:  since $|F|>\kappa$, $(V,\in)$ satisfies ``there exists an injection from $\kappa$ to $F$".  Since both $\kappa$ and $F$ are elements of $\mathfrak{M}$, it follows by elementarity that there is an injection $f: \kappa \to F$ with $f \in\mathfrak{M}$.  Since $\kappa \subset \mathfrak{M}$, for each $\alpha < \kappa$, $f(\alpha)$ is also an element of $\mathfrak{M}$.  So $\mathfrak{M}$ contains (as a subset) the range of $f$, which is a $\kappa$-sized subset of $F$.
 
 \item $(\mathfrak{M} \cap F) \setminus (\mathfrak{N} \cap F)$ is non-empty, because $|\mathfrak{M} \cap F|=\kappa > |\mathfrak{N}|$.

 \item $(\mathfrak{N} \cap F) \setminus (\mathfrak{M} \cap F)$ is non-empty.  To see this, note that since $|F|>\kappa = |\mathfrak{M} \cap F|$, the set difference $F \setminus \mathfrak{M}$ is non-empty.  Since $\mathfrak{N} \prec_* (V,\in)$, $\mathfrak{N} \models$ ``$F \setminus \mathfrak{M}$ is non-empty".  So there is some $x \in \mathfrak{N} \cap F$ with $x \notin \mathfrak{M}$.

\end{itemize}

Hence, there exists a pair $x,y \in F$ with  $x \in (\mathfrak{M} \cap F) \setminus (\mathfrak{N} \cap F)$ and $y \in (\mathfrak{N} \cap F)\setminus (\mathfrak{M} \cap F)$.  Since $F$ is locally cyclic there are $z \in F$ and $s,t \in S$ such that
\begin{align*}
zs &=x   \in (\mathfrak{M} \cap F) \setminus (\mathfrak{N} \cap F) \\
zt &=y  \in (\mathfrak{N} \cap F) \setminus (\mathfrak{M} \cap F).
\end{align*}

Now $x$ and $y$ are both in $P$, so by \eqref{eq_P_to_F_pure}, there is a $z_P \in P =(\mathfrak{M} \cap F) \cup (\mathfrak{N} \cap F)$ such that 
\begin{align}
z_P s &= x \in (\mathfrak{M} \cap F) \setminus (\mathfrak{N} \cap F) \label{eq_x}\\
z_P t &= y \in (\mathfrak{N} \cap F) \setminus (\mathfrak{M} \cap F). \label{eq_y}
\end{align}
Now $z_P$ is either in $\mathfrak{N} \cap F$ or $\mathfrak{M} \cap F$; in the former case equation \eqref{eq_x} violates closure of $\mathfrak{N} \cap F$ under the action, and in the latter case equation \eqref{eq_y} violates closure of $\mathfrak{M} \cap F$ under the action.

\begin{remark}\label{rem_bound_loc_cyc}
A similar argument shows that if the class of pure monomorphisms is cofibrantly generated in the category of $S$-acts, then there is a bound on the size of the locally cyclic $S$-acts.  That result, proved initially by the author using Theorem \ref{thm_NewChar}, was subsequently strengthened by coauthors in \cite{Cox_et_al_CofGenPureMon} and replaced with a model-theoretic argument. 
\end{remark}

\subsection{Proof of Theorem \ref{thm_SF},  \ref{item_SF_LeftCollapseIndec}
$\implies$ \ref{item_SF_MonoCofibGen}}

Assume $S$ is left-collapsible,\footnote{For all $s,t \in S$ there is a $u \in S$ with $us=ut$.  This implies that $S$ is right-reversible.} $\mathcal{SF}$-Mono is closed under compositions in $\overline{\text{Act}}$-$S$, and there is a bound on the size of the indecomposable members of $\mathcal{SF}$.  Closure of $\mathcal{SF}$-Mono under compositions implies (via the proof of \cite[Theorem 3.11]{MR3249868}) that $\mathcal{SF}$-Mono is cofibrantly closed; so Theorem \ref{thm_NewChar} is applicable.

Suppose $\kappa$ is a regular uncountable cardinal larger than $|S|$ and such that every indecomposable strongly flat act has size $<\kappa$.  We claim $\mathcal{SF}$-mono is $\kappa$-a.e.\ effective, which will imply, via the \ref{item_AE} $\implies$ \ref{item_CellGen} direction of Theorem \ref{thm_NewChar}, that $\mathcal{SF}$-mono is cellularly generated in $\overline{\text{Act}}$-$S$.

Suppose 
\[
S \cup \{ S \} \subset \mathfrak{N} \prec_* (V,\in)
\]
with $\mathfrak{N} \cap \kappa$ transitive, and suppose $f: A \subset B$ is a member of $\mathcal{SF}$-mono, i.e., $B/A$ strongly flat, such that $f \in \mathfrak{N}$.  We need to prove that $f \restriction \mathfrak{N}$ and $r^{f,\mathfrak{N}}$ are both in $\mathcal{SF}$-mono.  Similarly to the proof of Theorem \ref{thm_Fmono_cofibgen}, it will suffice to prove that, letting $F:=B/A$, both $\mathfrak{N} \cap F$ and $\frac{F}{\mathfrak{N} \cap F}$ are strongly flat.  By elementarity of $\mathfrak{N}$, there is an index set $I$ and a partition 
\[
F = \bigsqcup_{i \in I} F_i
\]
of $F$ into indecomposable (strongly flat) subacts such that the sequence $\langle F_i \ : \ i \in I \rangle$ is an element of $\mathfrak{N}$.  Moreover, elementarity of $\mathfrak{N}$ ensures that if $x \in F \cap \mathfrak{N}$ then the unique $i_x \in I$ such that $x \in F_{i_x}$ is an element of $\mathfrak{N}$.  In contrapositive form:
\begin{equation}\label{eq_outside_N}
i \in I \setminus \mathfrak{N} \  \ \implies \ \mathfrak{N} \cap F_i = \emptyset.
\end{equation}

Now if $i \in \mathfrak{N} \cap I$, elementarity of $\mathfrak{N}$ ensures $F_i \in \mathfrak{N}$; and since $F_i$ is indecomposable and strongly flat, its cardinality is $<\kappa$ by assumption.  Since $\mathfrak{N} \cap \kappa$ is transitive and $F_i$ is a $<\kappa$-sized element of $ \mathfrak{N}$, Fact \ref{fact_element_subset} ensures that $F_i \subset \mathfrak{N}$.  We have shown:
\begin{equation}\label{eq_i_in_N}
i \in \mathfrak{N} \cap I \ \implies \ F_i \subset \mathfrak{N} \cap F.
\end{equation}

Then \eqref{eq_outside_N} and \eqref{eq_i_in_N} imply that 
\[
\mathfrak{N} \cap F = \bigsqcup_{i \in \mathfrak{N} \cap I} F_i \text{ and } F \setminus (\mathfrak{N} \cap F) = \bigsqcup_{ i \in I \setminus \mathfrak{N}} F_i.
\]
In particular, $F \setminus  (\mathfrak{N} \cap F)$ is a subact of $F$.

Since strong flatness is preserved by disjoint unions, the left equality immediately implies $\mathfrak{N} \cap F$ is strongly flat.  And the right equality implies
\begin{equation}\label{eq_DisjointUnionFlat}
\frac{F}{\mathfrak{N} \cap F} \simeq \Theta_S \sqcup \bigsqcup_{i \in I \setminus \mathfrak{N}} F_i,
\end{equation}
where $\Theta_S$ denotes the single-point act.  Since $S$ is left-collapsible, $\Theta_S$ is strongly flat (\cite[III Corollary 13.7 and Exercise 14.3(4)]{MR1751666}).  So by closure of $\mathcal{SF}$ under disjoint unions, the quotient in \eqref{eq_DisjointUnionFlat} is strongly flat.  So both $f \restriction \mathfrak{N}$ and $r^{f,\mathfrak{N}}$ are in $\mathcal{SF}$-Mono, completing the proof that $\mathcal{SF}$-Mono is $\kappa$-a.e.\ effective.

\section{Proof of Theorem \ref{thm_UnitaryFlat} }\label{sec_Unitary}

Fix a monoid $S$ and a class $\mathcal{X}$ of $S$-acts with the property that for any collection $X_i$ of $S$-acts,
\[
\bigsqcup_{i \in I} X_i \in \mathcal{X} \ \iff \ \forall i \in I \ X_i \in \mathcal{X}.
\]
Recall $\mathcal{U}_{\mathcal{X}}$ denotes the class of monomorphisms $f: A \to B$ such that $B \setminus f(A)$ is a member of $\mathcal{X}$.  

Suppose first that $\mathcal{U}_{\mathcal{X}}$ is cofibrantly generated by some set $\mathcal{U}_0$.  Let $\mathcal{C}:= \text{cell}(\mathcal{U}_0)$.  By Theorem \ref{thm_NewChar} there is a $\kappa > |S|$ and a parameter $p$ witnessing 
\begin{equation}\label{eq_UnitC_reflect}
\left( \text{aa}^*_\kappa \mathfrak{N} \right) \left( \forall f \in \mathfrak{N} \cap \mathcal{C} \right)\left(  f \restriction \mathfrak{N} \in \mathcal{C} \text{ and } r^{f,\mathfrak{N}} \in \mathcal{C}\right).
\end{equation}

We claim that all indecomposable members of $\mathcal{X}$ have size $<\kappa$.  Suppose not; let $X$ be an indecomposable member of $\mathcal{X}$ of size at least $\kappa$, and consider $f: \emptyset \to X$.  By Fact \ref{fact_LS_class}, there is a $\mathfrak{N}$ of size $<\kappa$ such that $p,f \in \mathfrak{N} \prec_* (V,\in)$ and $\mathfrak{N} \cap \kappa \in \kappa$.  Then by \eqref{eq_UnitC_reflect}, $r^{f,\mathfrak{N}} \in \mathcal{U}_{\mathcal{X}}$.  Now $r^{f,\mathfrak{N}}$ is just the inclusion from $\mathfrak{N} \cap X \to X$; so membership of this inclusion in $\mathcal{C}$ ($\subseteq \mathcal{U}_{\mathcal{X}}$) implies $X \setminus (\mathfrak{N} \cap X)$ is closed under the $S$-action.  Since $|\mathfrak{N}|<\kappa \le |X|$, $X \setminus (\mathfrak{N} \cap X)$ is nonempty.  And elementarity of $\mathfrak{N}$ implies $\mathfrak{N} \cap X$ is also nonempty.  So $X$ is decomposable, contrary to assumptions.

\begin{remark}
We worked in $\overline{\text{Act}}$-$S$, but we could have also wrked in Act-$S$; instead of $\emptyset \to X$ one could work instead with the coproduct injection of $X$ into (say) the left part of $X \sqcup X$ (which is a member of $\mathcal{U}_{\mathcal{X}}$).
\end{remark}

For the other direction of Theorem \ref{thm_UnitaryFlat}, suppose $\mathcal{X}$ is closed under disjoint unions and that there is a bound, say $\kappa$, on the cardinalities of indecomposable members of $\mathcal{X}$.  Without loss of generality assume $\kappa > |S| + \aleph_0$.  We claim that $\mathcal{U}_{\mathcal{X}}$ is $\kappa$-a.e.\ effective, which by Theorem \ref{thm_NewChar} will imply cellular generation of $\mathcal{U}_{\mathcal{X}}$.

Assume
\[
S \cup \{ S \} \subset \mathfrak{N} \prec_* (V,\in)
\]
and $\mathfrak{N} \cap \kappa$ is transitive.  Fix any $f: A \to B$ in $\mathfrak{N} \cap \mathcal{U}_{\mathcal{X}}$.  Without loss of generality assume $f$ is an inclusion.  So 
\[
B = A \sqcup X
\]
and $X \in \mathcal{X}$.  By elementarity of $\mathfrak{N}$ and the assumption that $f \in \mathfrak{N}$  it follows that $X \in \mathfrak{N}$, $\mathfrak{N} \cap X$ is closed under the $S$-action, and
\[
\mathfrak{N} \cap B = (\mathfrak{N} \cap A) \sqcup (\mathfrak{N} \cap X). 
\]

The rest of the argument resembles the \eqref{item_SF_LeftCollapseIndec} $\implies$ \eqref{item_SF_MonoCofibGen} direction of the proof of Theorem \ref{thm_SF}.  By elementarity of $\mathfrak{N}$, there is a partition $\vec{X}=\langle X_i \ : \ i \in I \rangle$ of $X$ into connected components, with $\vec{X} \in \mathfrak{N}$.  The assumption about the class $\mathcal{X}$ implies that each $X_i$ is a member of $\mathcal{X}$.  So each $X_i$ is is an indecomposable member of $\mathcal{X}$ and hence $|X_i| < \kappa$ for each $i \in I$.  It follows (as in the argument of Theorem \ref{thm_SF}) that 
\[
\mathfrak{N} \cap X = \bigsqcup_{i \in \mathfrak{N} \cap I}  X_i \text{ and } X \setminus (\mathfrak{N} \cap X) = \bigsqcup_{i \in I \setminus \mathfrak{N}} X_i.
\]
Then closure of $\mathcal{X}$ under disjoint unions yields
\begin{equation}\label{eq_N_cap_X_in_X}
\mathfrak{N} \cap X \in \mathcal{X}
\end{equation}
and
\begin{equation}\label{eq_CompNcapX}
X \setminus (\mathfrak{N} \cap X) \in \mathcal{X}.
\end{equation}

The diagram $\mathcal{D}^{f,\mathfrak{N}}$ looks like this, with all arrows inclusions:
\[\begin{tikzcd}
	A && {B=A\sqcup X} \\
	& \begin{array}{c} A \cup (\mathfrak{N} \cap B) \\=A \sqcup (\mathfrak{N} \cap X) \end{array} \\
	{\mathfrak{N} \cap A} && {\mathfrak{N} \cap B = (\mathfrak{N} \cap A) \sqcup (\mathfrak{N} \cap X)}
	\arrow[hook, from=1-1, to=1-3]
	\arrow[curve={height=12pt}, hook, from=1-1, to=2-2]
	\arrow["{r^{f,\mathfrak{N}}}"', hook, from=2-2, to=1-3]
	\arrow["\ulcorner"{anchor=center, pos=0.125, rotate=-90}, draw=none, from=2-2, to=3-1]
	\arrow[hook, from=3-1, to=1-1]
	\arrow["{f \restriction \mathfrak{N}}"', hook, from=3-1, to=3-3]
	\arrow[hook, from=3-3, to=1-3]
	\arrow[curve={height=-12pt}, from=3-3, to=2-2]
\end{tikzcd}\]

Then $f \restriction \mathfrak{N}$ is in $\mathcal{U}_{\mathcal{X}}$ because of \eqref{eq_N_cap_X_in_X} and $r^{f,\mathfrak{N}}$ is in $\mathcal{U}_{\mathcal{X}}$ because of \eqref{eq_CompNcapX}.

\section{Some Questions}\label{sec_Questions}

Recall Corollary \ref{cor_Positive}: if $S$ is right reversible and flat acts are closed under stable Rees extensions, then FCC holds for Act-$S$.  We conjecture that the closure of flat acts under stable extensions is superfluous.  In other words:

\begin{conjecture}
FCC holds for Act-$S$ whenever $S$ is right-reversible.
\end{conjecture}

If $S$ is right reversible and $\mathcal{F}_S$ is \emph{not} closed under stable extensions,\footnote{We actually are not aware of such examples, but suspect they exist.} then $\mathcal{F}_S$-Mono is not cellularly closed in $\overline{\text{Act}}$-$S$; it's not even closed under ordinary composition.  But $\mathcal{F}_S$-PureMono (\emph{pure} monomorphisms with flat Rees quotient) is always closed under composition (\cite[Theorem IV.1.5]{renshaw1986flatness}) and is in fact cellularly closed in $\overline{\text{Act}}$-$S$.  If it is cellularly generated in $\overline{\text{Act}}$-$S$, then FCC holds in Act-$S$, by Lemma \ref{lem_CellGenExpandedGivesCovers}.  Hence the following is a stronger conjecture than the one above:

\begin{conjecture}\label{conj_RR_FPureMono}
If $S$ is right-reversible then $\mathcal{F}$-PureMono is cellularly generated.
\end{conjecture}

One can attempt to prove Conjecture \ref{conj_RR_FPureMono} in a manner similar to the proof of Theorem \ref{thm_LO}.  Fix $\kappa > |S|+\aleph_0$ and attempt to prove $\mathcal{F}$-PureMono is $\kappa$-a.e.\ effective.  Suppose $f: A \to B$ is in $\mathcal{F}$-PureMono, and that 
\[
S \cup \{ S, f \} \subset \mathfrak{N} \prec_* (V,\in),
\]
with $\mathfrak{N} \cap \kappa$ transitive.  We need to prove that $f \restriction \mathfrak{N}$ and $r^{f,\mathfrak{N}}$ are both in $\mathcal{F}$-PureMono.  This is straightforward for $f \restriction \mathfrak{N}$, using elementarity of $\mathfrak{N}$.  And, as with the proof of Theorem \ref{thm_Fmono_cofibgen}, $r^{f,\mathfrak{N}}$ is a member of $\mathcal{F}$-Mono, and hence in $\mathcal{F}$-StableMono (since its Rees quotient is flat, see discussion after the statement of Theorem \ref{thm_Fmono_cofibgen}).  I.e., $r^{f,\mathfrak{N}}$ is stable, with flat Rees quotient.  But we do not know if $r^{f,\mathfrak{N}}$ is pure.  In certain situations we can conclude it, such as our Theorem \ref{thm_LO}, where we assumed $S$ is right-LO.  If $S$ is \emph{not} right-LO, it is known (\cite{Cox_et_al_CofGenPureMon}) that PureMono is \emph{not} cofibrantly generated in $\overline{\text{Act}}$-$S$.  So in that situation one couldn't simply use the same ``intersection" trick 
\[
\mathcal{F}\text{-PureMono} = \text{PureMono} \ \cap \ \mathcal{F} \text{-Mono}
\]
we used to prove Theorem \ref{thm_LO}.  But this still leaves open the possibility that $\mathcal{F}$-PureMono might be cofibrantly generated even for some right-reversible $S$ that aren't right-LO.

Our main tool for all the results, Theorem \ref{thm_NewChar}, says that a cellularly-closed class $\mathcal{M}$ of monomorphisms is cellularly generated if and only if $\mathcal{M}$ is almost-everywhere effective,  which is a set-theoretic notion.  
\begin{question}
Translate the ``almost everywhere effective" language of Definition \ref{def_EffectiveSubobj}, and the corresponding Theorem \ref{thm_NewChar}, into category-theoretic language.  
\end{question}

\noindent Theorem 3.1 of \cite{MR4594300} comes close, but only under a continuity assumption on the class $\mathcal{M}$ that fails, for example, in \textbf{Ab} when $\mathcal{M}$ is the (cofibrantly-closed) class of abelian group monomorphisms with free quotient.

\end{document}